\documentclass[10pt]{elsarticle} % 
\usepackage{amssymb}
\usepackage{amssymb}
\usepackage{amsmath}
\usepackage{amsfonts} 
\usepackage{amsbsy}
\usepackage{amscd}
\usepackage{amsthm}
\usepackage{fullpage}
\usepackage{times}
\usepackage{psfrag}
\usepackage{graphics}
\usepackage{color}
\usepackage{bbm}
\usepackage{colortbl}
\usepackage{graphicx}
\usepackage{caption}
\usepackage{bbm}
\usepackage{subcaption}
\usepackage{array,supertabular}
\usepackage{tabularx}
\usepackage{booktabs}
\usepackage{multirow}
\usepackage{ulem}
\usepackage{units}
\usepackage{mathabx}
\usepackage{accents}
\usepackage{makecell}
\usepackage{float}
\usepackage{relsize}
\usepackage{xfrac}
\usepackage{algpseudocode}
\usepackage{rotating}
\usepackage[capitalize]{cleveref}
\usepackage{mathrsfs} %better calligraphic fonts (\mathscr instead of \mathcal)
\usepackage[framemethod=tikz]{mdframed} %%to make a framed environment
\usepackage{circuitikz}
\usepackage{makecell}
\usetikzlibrary{arrows,matrix,positioning}
\usepackage{subcaption}
\usepackage[framemethod=tikz]{mdframed}
\usepackage[most]{tcolorbox}
\usepackage{lipsum}% just to generate text for the example
\usetikzlibrary{shapes.geometric}

\usepackage{appendix}

\crefname{secinapp}{Section}{Sections}
\Crefname{secinapp}{Section}{Sections}

\newlength{\dhatheight}

\newtheorem{theorem}{Theorem}[section]
\newtheorem{corollary}[theorem]{Corollary}

\newtheorem{conjecture}[theorem]{Conjecture}
\newtheorem{definition}[theorem]{Definition}
\newtheorem{remark}[theorem]{Remark}
\newtheorem{lemma}[theorem]{Lemma}
\newtheorem{example}[theorem]{Example}

\theoremstyle{definition}
\newtheorem{algorithm}[theorem]{Algorithm}
\theoremstyle{remark}
\newmdtheoremenv[
hidealllines=true,
leftline=true,
innertopmargin=0pt,
innerbottommargin=0pt,
linewidth=4pt,
linecolor=gray!40,
innerrightmargin=0pt,
innertopmargin=-6pt,
]{examplei}{Example} 

\begin{document}

\newcommand{\js}[1]{#1_s}

% Begin document, frontmatter is first

\begin{frontmatter}
\title{Resistance values under transformations in regular triangular grids
}

\author[byu1]{Emily. J. Evans} 
\author[tu]{Russell Jay Hendel} 

\address[byu1]{Department of Mathematics,
  Brigham Young University,
  Provo, Utah 84602, USA}
\address[tu]{Dept of Mathematics, Towson University, Towson Maryland 21252, USA}
 
\begin{abstract}
In ~\cite{evans2021algorithmic,Hendel} the authors investigated resistance distance in triangular grid graphs and observed several types of asymptotic behavior. This paper extends their work by studying the initial, non-asymptotic, behavior found when equivalent circuit transformations are performed, reducing the rows in the triangular grid graph one row at a time.  The main conjecture characterizes, after reducing an arbitrary number of times an initial triangular grid all of whose edge resistances are identically one, when   edge resistance values are less than, equal to, or greater than one. A special case of the conjecture is proven. The main theorem identifies patterns of repeating edge resistances arising in diagonals of a triangular grid reduced $s$ times provided the original grid has at least $4s$ rows of triangles.   This paper also improves upon the notation, concepts, and proof techniques introduced by the authors previously. \end{abstract}


\end{frontmatter}

\section{Introduction}

% Note \#3: Someone has to write a cover letter. I can volunteer for that (place it in overleaf so you can look at it.
% \color{red} Great! I'm happy if you do it.
\color{black}

 In  \cite{Barrett9}, a paper devoted to the family of straight linear 2--trees,  the triangular grid graph was introduced, which heuristically consists of $n$ rows
of upright-oriented equilateral triangles with edge weight and resistance equal to one. The authors conjectured an explicit asymptotic formula.   As a result of empirical numerical calculations using the combinatorial Laplacian, rather than circuit analysis, the authors conjectured an explicit asymptotic formula connected with the triangular grid.   Moreover~\cite{Barrett9, evans2021algorithmic} introduced  an algorithm for circuit reduction of the triangular grid graph, consisting
of reducing the triangular graph one row at a time, but the authors made no attempt to determine resistance value patterns in the intermediate reduced graphs.

The computational use of the algorithm for circuit reduction of the triangular grid was first studied in ~\cite{Hendel}.   This approach provides additional numerical data strongly supporting several conjectures.  For example,
as $n$ goes to infinity, the resistance of the edges of the single-triangle graph resulting from reducing the triangular graph with $n$ rows 
$n-1$ times unexpectedly approaches $\frac{3}{2e}$ in limit. 

This paper, complements the study of asymptotic behavior  in~\cite{Barrett9} and~\cite{Hendel} by studying non-asymptotic behavior arising in reduced grids.  The main conjecture  characterizes, after reducing an arbitrary number of times an initial triangular grid all of whose resistance values are initially equal to one, when resistance values are less than, equal to, or greater than one.  A special case of the conjecture is proven. The main theorem identifies patterns of repeating edge labels arising in diagonals of a triangular grid reduced $s$ times provided the original grid had at least $4s$ rows of triangles.   This paper also improves upon the notation, concepts, and proof techniques introduced by the authors previously.

An outline of this paper is as follows. In Section~\ref{sec:def} we review previous applications of circuit theory to the triangular grid as well as previous notation and concepts which are extended as needed.
In Section \ref{sec:motivation} we present a simply-formulated conjecture identifying  edge resistances that are less than, equal to, or greater than one after reducing an arbitrary number of times an initial triangular grid with uniform labels of ones.  This conjecture motivated the result in Section \ref{sec:constcenter} about the repeating edge resistances found in the middle portions of reduced $n$-grids and in fact allows, in Section \ref{sec:edgeval}, the proof of a special case of the main conjecture. Section \ref{sec:constcenter} is preceded by
Section~\ref{sec:proofmethods}, which reviews and simplifies proof methods from \cite{Hendel} and Section~\ref{sec:illandconseq}, which illustrates the techniques of Section~\ref{sec:proofmethods}. The paper concludes with Section \ref{sec:open} which presents a simple numerical pattern in these reduced $n$-grids, illustrative of the plethora of patterns that apparently abound in these triangular grids and their reductions, pointing to a fertile hunting ground for new patterns, techniques, and proof methods.

\section{Definitions, Algorithms, Conventions and Notations}\label{sec:def}
Much of the following comes from \cite{Hendel} and this attribution is not repeated at each item. However,  where an idea
is extended this is made explicit.
More specifically, this section introduces

\begin{itemize}
\item Circuit transformations, and the $\Delta, Y$ functions,
\item The definition of the $n$-triangular-grid,
\item Conventions about upright triangles,
\item Notations for triangles and edges,
\item The reduction algorithm,
\item Notation for $Y$-legs and reduced triangles,
\item Definitions and notational conventions about symmetry and equality involving grids,
\item $d$-rims and the \textit{the upper left half},
\item Subgrids, boundaries, and notation.
\end{itemize}

\subsubsection*{Circuit transformations.} 
Throughout the paper we use the following well-known circuit transformations \begin{definition}[Series Transformation] Let $N_1$, $N_2$, and $N_3$ be nodes in a graph where $N_2$ is adjacent to only $N_1$ and $N_3$.  Moreover, let $R_A$ equal the resistance between nodes $N_1$ and $N_2$ and $R_B$ equal the resistance between nodes $N_2$ and $N_3$.  A series transformation transforms this graph by deleting $N_2$ and setting the resistance between $N_1$ and $N_3$ equal to $R_A + R_B$.\end{definition}

%\begin{definition}[Parallel Transformation] Let $N_1$ and $N_2$ be nodes in a multi-edged graph where $e_1$ and $e_2$ are two edges between $N_1$ and $N_2$ with resistances $R_A$ and $R_B$, respectively.   A parallel transformation transforms the graph by deleting edges $e_1$ and $e_2$ and adding a new edge between $N_1$ and $N_2$ with edge resistance $r = \left(\frac{1}{R_A} + \frac{1}{R_B}\right)^{-1}$.
%\end{definition}

A $\Delta$--Y transformation is a mathematical technique to convert three resistors labeling 3 edges of a 3-vertex graph  (in a $\Delta$ formation) to an equivalent system of three resistors labeling 3 edges of a 4-vertex graph ( in a ``Y'' formation) as illustrated in Figure~\ref{fig:dy}.  We formalize this transformation below. 
\begin{definition}[$\Delta$--Y transformation]\label{def:dy}
Let $N_1, N_2, N_3$ be nodes and $R_A$, $R_B$ and $R_C$ be given resistances as shown in Figure~\ref{fig:dy}.  The transformed circuit in the ``Y'' formation as shown in Figure~\ref{fig:dy} has the following resistances:
\begin{align}
  R_1 &= \frac{R_BR_C}{R_A + R_B + R_C} \\
  R_2 &= \frac{R_AR_C}{R_A + R_B + R_C} \\
  R_3 &= \frac{R_AR_B}{R_A + R_B + R_C}.
\end{align}
If $x,y$ and $z$ are the resistances of sides of a triangle then we introduce the function
 $ \Delta(x,y,z) =\frac{xy}{x+y+z}$.
 By convention, the first two arguments of $\Delta$ will be the labels of the two adjacent edges whose common vertex contains the $Y$-leg of the resulting computation.
\end{definition}
\begin{definition}[Y--$\Delta$ transformation]\label{def:yd}
Let $N_1, N_2, N_3$ be nodes and $R_1$, $R_2$ and $R_3$ be given resistances as shown in Figure~\ref{fig:dy}.  The transformed circuit in the ``$\Delta$'' format as shown in Figure~\ref{fig:dy} has the following resistances:
\begin{align}
  R_A &= \frac{R_1R_2+R_1R_3+R_2R_3}{R_1} \\
  R_B &= \frac{R_1R_2+R_1R_3+R_2R_3}{R_2} \\
  R_C &= \frac{R_1R_2+R_1R_3+R_2R_3}{R_3}.
\end{align}
If $x,y$ and $z$ are resistances of a $Y$ we define the function
$Y(x,y,z)=\frac{xy+yz+zx}{x}.$ 
By convention, the first argument of the $Y$ function will be the resistance of the $Y$ leg  not incident with the $\Delta$-edge begin computed.   
\end{definition}
 \begin{figure}
 \begin{center}
\resizebox{4in}{!}{ \begin{circuitikz}
 
 \draw (0,0) to [R, *-*,l=$R_B$] (4,0);
  \draw (0,0) to [R, *-*,l=$R_C$] (2,3.46);
  \draw (2,3.46) to[R, *-*,l=$R_A$] (4,0)
  {[anchor=north east] (0,0) node {$N_1$} } {[anchor=north west] (4,0) node {$N_3$}} {[anchor=south]  (2,3.46) node {$N_2$}};
  
  \draw (6,0) to [R, *-*,l=$R_1$] (8,1.73);
  \draw (6,3.46) to [R, *-*,l=$R_2$] (8,1.73);
 \draw (10.5,1.73) to [R, *-*,l=$R_3$] (8,1.73)
  {[anchor=north east] (6,0) node {$N_1$} } {[anchor=north west] (10.5,1.73) node {$N_3$}} {[anchor=south]  (6,3.46) node {$N_2$}};
 %   \draw (4,0) to [R] (4,4);
 %    --
 %  to[R] (1,1)   to[R] (0,2) --
 % (3,1) to[R] (2,2)
 ; 
 \end{circuitikz}}
 \end{center}
 \caption{ $\Delta$ and $Y$ circuits with vertices labeled as in Definition~\ref{def:dy}.}
 \label{fig:dy}
 \end{figure}
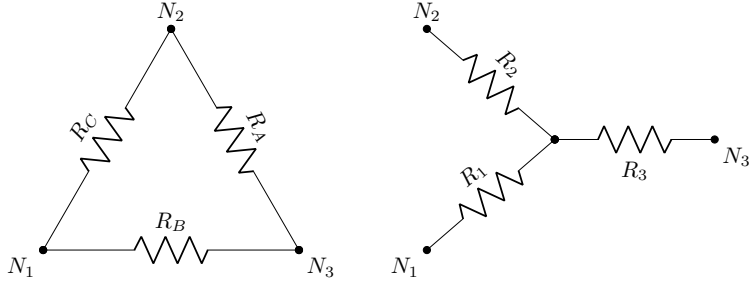
% \begin{proposition} Series transformations, parallel transformations, and $\Delta$--Y transformations yield equivalent circuits.
% \end{proposition}
% \begin{proof}
% See~\cite{oldbook} for a proof of this result.
% \end{proof}
% functions used to analyze electrical circuits.
% If $x,y$ and $z$ are the resistances of sides of a triangle then we let
%  $ \Delta(x,y,z) =\frac{xy}{x+y+z}$, Similarly, if $x,y$ and $z$ are resistances of a $Y$ then we let
% $Y(x,y,z)=\frac{xy+yz+zx}{x},$

\subsubsection*{Definition of $n$-grid.} There are a variety of ways to define graphs including using
an inductive approach, by general properties of graphs, and by
the adjacency-matrix. We present a simple definition by vertices and
edges.

\begin{definition}\label{def:ngrid}
An $n$-triagular-grid (usually abbreviated as an $n$-grid)
%or an $n-T-grid$ 
is any graph that is (graph-) isomorphic to the  graph,
whose vertices are all integer pairs $(x,y)=(2r+s,s)$ 
in the Cartesian plane, with $r$ and $s$ integer
parameters satisfying
$0 \le r \le n, 0 \le s \le n-r;$
  and whose edges
consist of any two vertices $(x,y)$ and $(x',y')$ with
either i) $x'-x=1, y'-y=1,$ 
 ii) $x'-x=2, y'-y=0,$ or
iii) $x'-x=1, y'-y=-1.$  See Figure~\ref{fig:3grid}.
\end{definition}

%Figure \ref{fig:3grid} illustrates the 3-grid.

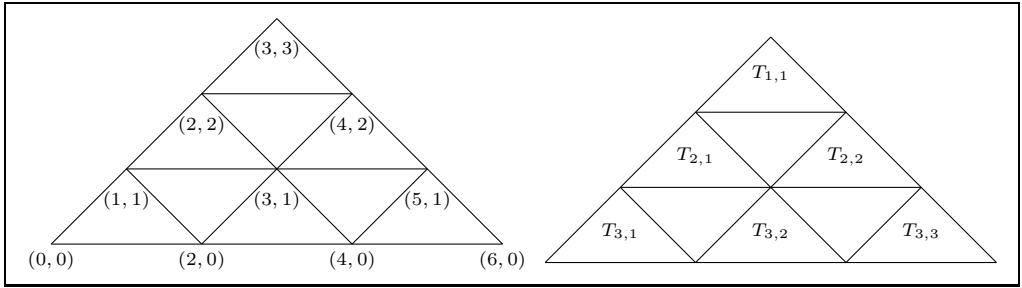
\begin{figure} 
\begin{center}
\begin{tabular}{|c|}
\hline
\begin{tikzpicture} [xscale=1,yscale =1]

\node [below] at (3,3.2) {    }; 
 
\draw (0,0)--(1,1)--(2,2)--(3,3);
\draw (2,0)--(3,1)--(4,2);
\draw (4,0)--(5,1);

\draw (3,3)--(4,2)--(5,1)--(6,0);
\draw (2,2)--(3,1)--(4,0);
\draw (1,1)--(2,0);

\draw (0,0)--(2,0)--(4,0)--(6,0);
\draw (1,1)--(3,1)--(5,1);

\draw (2,2) -- (4,2);

%added Nodes
\begin{scriptsize}
\node [below] at (3,2.8) {$(3,3)$};

\node [below] at (2,1.8) {$(2,2)$};
\node [below] at (4,1.8) {$(4,2)$};

\node [below] at (1,.8) {$(1,1)$};
\node [below] at (3,.8) {$(3,1)$};
\node [below] at (5,.8) {$(5,1)$};

\node [below] at (6,0) {$(6,0)$};
\node [below] at (4,0) {$(4,0)$};
\node [below] at (2,0) {$(2,0)$};
\node [below] at (0,0) {$(0,0)$};

\end{scriptsize}
\end{tikzpicture}

\begin{tikzpicture} [xscale=1,yscale =1]

\node [below] at (3,3.2) {    }; 
 
\draw (0,0)--(1,1)--(2,2)--(3,3);
\draw (2,0)--(3,1)--(4,2);
\draw (4,0)--(5,1);

\draw (3,3)--(4,2)--(5,1)--(6,0);
\draw (2,2)--(3,1)--(4,0);
\draw (1,1)--(2,0);

\draw (0,0)--(2,0)--(4,0)--(6,0);
\draw (1,1)--(3,1)--(5,1);

\draw (2,2) -- (4,2);

\begin{scriptsize}
\node  [above] at (1,.2) {$T_{3,1}$};  
\node  [above] at (3,.2) {$T_{3,2}$};
\node  [above] at (5,.2) {$T_{3,3}$};
\node  [above] at (2,1.2) {$T_{2,1}$};
\node  [above] at (4,1.2) {$T_{2,2}$};
\node  [above] at (3,2.3) {$T_{1,1}$};

\node [below] at (6,0) {$ $};
\node [below] at (4,0) {$ $};
\node [below] at (2,0) {$ $};
\node [below] at (0,0) {$ $};

%\node [right] at (0,-.75) {Panel B: A 3-grid   labeled by row and diagonal.};
\end{scriptsize}

\end{tikzpicture}
\\  \hline
\end{tabular}
\caption{A 3-grid embedded in the Cartesian Plane constructed using Definition~\ref{def:ngrid} (Left Panel) with the paper's conventions on row and diagonal coordinates (Right Panel).  Rows are labeled from top to bottom and diagonals are labeled from left to right as shown.  When necessary,  sides of triangles are indicated by the letters $L,R,B,$ standing for left, right, and base side.  Thus, $T_{2,1,B}$ is the base edge of the triangle in row 2 diagonal 1. Throughout the paper (e.g. in the description of the row reduction algorithm) we are only concerned with the upright-oriented triangles.}\label{fig:3grid}
\end{center}
\end{figure}

\color{black}

%\begin{notation}\label{ngrid1}  
%\textbf{Upright--oriented triangles.} Throughout the paper %we will find it useful to use heuristics. We perceive
%the $n$-grid as a triangular grid consisting of $n$
%rows of triangles. The $r$-th row of the $n$-grid has 
%$2r-1$ triangles of which $r$ are in the \emph{upright-%oriented} position and the remaining $r-1$ triangles are in %the downward oriented position
% as illustrated in Figure \ref{fig:3grid}.
%For purposes of describing
%the row-reduction step we are only concerned with the $r$ %upright-oriented
%equilateral triangles; consequently,
%throughout the paper when speaking about triangles, we will %assume the reference is to  upright-oriented triangles.
%\end{notation}

%\begin{notation} \label{ngrid2} 
%\textbf{Triangle and edge notation.} The following %terminology and notation will be used throughout the paper. %We will refer to the \emph{left}, \emph{right} and %\emph{base}\
%sides of the triangles. The left, right, and base sides of a %triangle will also
%be called the   1-edge, 2-edge, and 3-edge. 
%We let $\langle r,d \rangle_n, 1 \le d \le r,$ refer  to the %(upright oriented) triangle
%in row $r$ and diagonal $d$
% with diagonals enumerated left to right, and rows %enumerated top to bottom. If each 
%edge in an $n$-grid is labeled with an electrical resistance %(which is equivalent to the reciprocal of the edge weight), %then,
%$\langle r,d,e \rangle, e \in \{1,2,3\}$ refers to the %resistance of edge
%$e$ of triangle $\langle r,d \rangle.$
%\end{notation}

\subsubsection*{The Reduction Algorithm.} We recall the algorithm presented in ~\cite{evans2021algorithmic,Hendel}  for performing reductions on the $n$-grid as defined by Definition \ref{def:ngrid} and using
the notational conventions mentioned in the caption of Figure \ref{fig:3grid}.    The algorithm will take an $n$-grid and reduce the number of rows by 1, resulting in an $n-1$ grid. Since the algorithm uses circuit functions that produce equivalent electrical resistance this $n-1$ grid is constructed so that the resistance between two corner vertices remains the same.
\begin{algorithm}\label{alg:supertri}

\item Given an $n$-grid with labeled edges, the   
\emph{row-reduction} of this labeled $n$-grid (to a labeled $n-1$-grid) 
refers to the sequential performance of the following steps:
\begin{itemize}
\item (Step A) Identify the edge-labels of relevant triangles in the original $n$-grid
\item (Step B) Perform a $\Delta-Y$ transformation,
  on each (upright-oriented) triangle.  

\item (Step C) While we retain the 3 corner tails, that is, the edges with a degree 1 vertex (these are the dashed lines in Panel B) for possible later computation (see Remark \ref{rem:tails}), we discard them from the diagram for ease of understanding, since this discarding does not affect the resistance labels of the edges of the reduced grid displayed in Panel E

\item (Step D) Starting from any corner vertex (degree-2 node) and say going clockwise around the boundary of the Panel C, perform series transformations on every pair of incident edges.

\item (Step E) Perform $Y-\Delta$ transformations on the remaining non-boundary edges.  

\end{itemize}
\end{algorithm}

%\color{blue}
The algorithm is illustrated in Figure~\ref{fig:5panels}.

\begin{figure}[ht!]
\begin{center}
\begin{tabular}{||c|c|c|c|c||}
\hline
\begin{tikzpicture}[yscale=.5,xscale=.5]

\node [below] at (3,3.2) {    }; 
 
\draw (0,0)--(1,2)--(2,0)--(0,0);
\draw (2,0)--(3,2)--(4,0)--(2,0);
\draw (4,0)--(5,2)--(6,0)--(4,0); 
\draw (1,2)--(2,4)--(3,2)--(1,2);
\draw (3,2)--(4,4)--(5,2)--(3,2);
\draw (2,4)--(3,6)--(4,4)--(2,4);

\node [below] at (3,6.2) {     };
\node [below] at (3,-.5) {Panel A};
\end{tikzpicture} 
& 
\begin{tikzpicture}[xscale=.5,yscale=.5] 
\draw  (1,1)--(1,2)--(2,3)--(2,4)--(3,5);
\draw  (5,1)--(5,2)--(4,3)--(4,4)--(3,5);
\draw (1,1)--(2,0)--(3,1)--(3,2)--(2,3);
\draw (3,1)--(4,0)--(5,1)--(5,2)--(4,3)--(3,2); 
 
\draw [dashed] (6,0)--(5,1);
\draw [dashed] (0,0)--(1,1);
 \draw [dashed] (3,5)--(3,6);
 
\node [below] at (3,-.5) {Panel B};
\end{tikzpicture}&
\begin{tikzpicture}[yscale=.5,xscale=.5]
 
\draw [dashed] (1,1)--(1,2)--(2,3)--
(2,4)--(3,5);
\draw [dashed] (5,1)--(5,2)--(4,3)--(4,4)--(3,5);
\draw [dashed] (1,1)--(2,0)--(3,1)--(4,0)--(5,1);
\draw (3,1)--(3,2)--(2,3)--(3,2)--(4,3);

\node [below] at (3,-.5) {Panel C};
 
\end{tikzpicture}&
\begin{tikzpicture}[yscale=.5,xscale=.5]
 
\draw (1,1)--(2,3)--(3,5);
\draw (5,1)--(4,3)--(3,5);
\draw (1,1)--(3,1)--(5,1);
\draw [dotted] (2,3)--(3,2)--(4,3);
\draw [dotted] (3,2)--(3,1); 
 
 \node [below] at (3,-.5) {Panel D};
\end{tikzpicture}&
\begin{tikzpicture}[yscale=.5,xscale=.5]
 
\draw (1,1)--(2,3)--(3,5);
\draw (5,1)--(4,3)--(3,5);
\draw (1,1)--(3,1)--(5,1);
\draw (2,3)--(4,3)--(3,1)--(2,3); 
 
 \node [below] at (3,-.5) {Panel E};
 
\end{tikzpicture}
\\ \hline
\end{tabular}

\end{center}
\caption{Illustration of one row reduction, Algorithm \ref{alg:supertri}, on a 3-grid all of whose edge-labels are 1. The panel labels correspond to the five steps of  Algorithm \ref{alg:supertri} }\label{fig:5panels}
\end{figure}
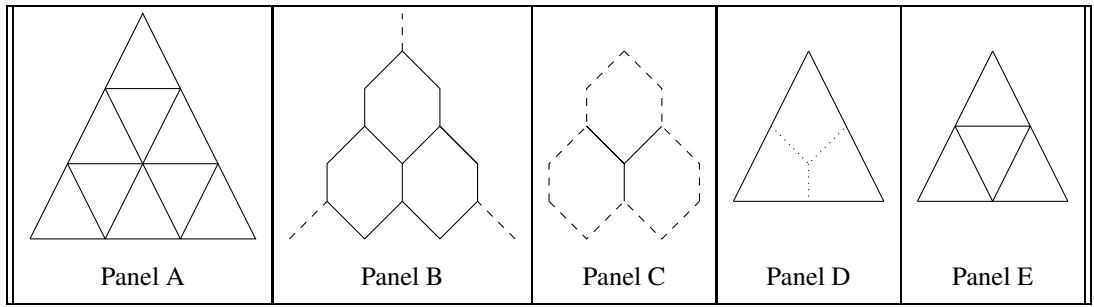
\color{black}
 
\begin{remark}\label{rem:tails}
The idea of retaining the corner tails for future computations, but omitting them in future reduction steps may seem unusual without appropriate context.  In ~\cite{Barrett9,evans2021algorithmic, Hendel} the problem of determining the resistance distance between the degree two vertices in the original graph was considered.  For this problem it is essential that these tail values be retained to determine the resistance distance. However, to reduce the 3 grid displayed in Panel A to the 2-grid displayed in Panel E, these corner tails are not needed. In fact, the determination of the resistance distance between these vertices led to the reduction algorithm and the results detailed in this paper.
\end{remark}

\begin{definition}\label{def:rr}
We call a single pass--through of Algorithm~\ref{alg:supertri} (i.e., performing steps A-E one time) a reduction.
\end{definition}

We next indicate more notational conventions used throughout the paper.
\begin{itemize}
\item We will use clock notation (without or without row and diagonal coordinates) to refer to the $Y$-legs arising from a $\Delta-Y$ transformation.  Thus     $Y_{12}, Y_{8}, Y_{4}$ refer to the legs of an
upside-down $Y;$ we abuse notation and let $Y_{12}, Y_{8}, Y_{4}$ also refer to the underlying resistances, the meaning always being clear from context. Thus, $Y_{12,2,1}$ indicates the resistance of the $Y_{12}$ edge resulting from applying a $\Delta-Y$ transform to the triangle $T_{2,1}.$
\item We let $T^k$ refer to the $k$-grid all of whose edge resistances are 1. $T^{k,k'}$ is the $k'$-grid obtained from $T(k)$ by applying $k-k'$ reduction steps.
\item For any $k,k',$ $T^{k,k'}$ has various symmetries. Symmetry, in this paper, refers to vertical, slide,  and rotational symmetry
(by $\frac{\pi}{3}$ and $\frac{2\pi}{3})$ of the underlying grid regarded
as a labeled graph (See Definition \ref{def:gridsymmetry})
\item When performing a single reduction from $T^{k,k-i+1}$ to
$T^{k,k-i}$ it may be necessary to distinguish between triangles in the parent grid, $T^{k,k-i+1}$ and child grid $T^{k,k-i}.$ In such a case we will use double subscript or superscript notation e.g.   $T^{k,k-i}_{r,d}$ is the triangle in grid $T^{k,k-i}$ in row $r$ and
diagonal $d.$ 
\end{itemize}
\begin{definition}\label{def:pretype} A statement of the form $T_{r,d} = T_{r',d'}$ means 
$T_{r,d,X}=T_{r',d',X}, X \in \{L,R,B\}.$ 
The statement $T^{k,k'
} =1$  by convention means $T^{k,k'}_{r,d,X}=1, 1 \le r \le k', 1 \le d \le r, X \in \{R, L, B\}. $
The \textit{pre-type} of a triangle $T_{r,d}$ is the ordered triple $(T_{r,d,L}, T_{r,d,R}, T_{r,d,B}).$ Two triangles have the same \textit{type},   if their pre-types are identical under symmetry. When indicating types we can use any pre-type  ordering associated with it.
\end{definition}

\subsubsection*{$d$-rims, subgrids, and the upper left half.} 
Hendel \cite[Equations (29),(30)]{Hendel} introduced the idea of perceiving the $n$-grid as a collection of concentric triangular annuli or triangular rims. He also introduced the idea of the \textit{upper left half} of the $n$-grid.
Prior to defining these concepts, we present a coordinate independent definition of the symmetries.

 \begin{definition}[Grid Symmetry] \label{def:gridsymmetry}
 \cite[Definition 9.1]{Hendel}
 When discussing $n$-grids,
by vertical symmetry we mean that for $1 \le r \le n, 1 \le d \le r,$ 
$$
 T_{r,d,L} = T_{r,r+1-d,R}, \qquad
 T_{r,d,R} = T_{r,r+1-d,L}, \qquad
 T_{r,d,B} = T_{r,r+1-d,B},
$$
 by rotational symmetry, we mean that for $1 \le r \le n, 1 \le d \le r,$
$$
T_{r,d,L} = 
T_{n+d-r,n+1-r,R}, \qquad
T_{r,d,R} = 
T_{n+d-r,n+1-r,B}, \qquad
T_{r,d,B} = 
T_{n+d-r,n+1-r,L},
$$
and by {slide} symmetry we mean  
$$
T_{r,d,L} = 
T_{n+d-r,d,L}, \qquad
1 \le d \le r \le n.
$$
\end{definition}

It is straightforward to verify that an assumption of vertical and slide symmetry is equivalent to an assumption of vertical and rotational symmetry. Note also, that if the $n$-grid is represented by $n$-rows of equilateral triangles then the above definitions of vertical and rotational symmetry coincide with the coordinate definitions of vertical symmetry and clockwise rotation by $\frac{\pi}{3}.$
 
\begin{definition}[The Upper Left Half]\label{def:upperlefthalf} Given integers $n$ and $m$ with $0 \le n \le m-1,$
the \emph{upper left half} of the $n$ grid, $T^{m,n} $ arising from $m-n$ reductions of $T^m,$ is the set of triangles
\begin{equation}\label{equ:Upperlefthalf}
    S = \left\{T_{r,d}: 
         1 \le d \le \left\lfloor \frac{n+2}{3} \right\rfloor,
        2d-1 \le r \le \left\lfloor \frac{n+d}{2} \right\rfloor\right\}.
\end{equation}
\end{definition}
 
 \begin{example}\label{exa:upperlefthalf}
 If $n=3,$ (see panel A in Figure \ref{fig:5panels}) the upper left half consists of the 
 triangles $T_{r,d}, d=1, r=1,2.$
 If $n=7$ (see Figure \ref{fig:7gridrims}) the upper left half consists of the union of the triangle sets 
 $T_{r,d}. d=1, 1 \le r \le 4,$
 $T_{r,d}, d=2, 3 \le r \le 4,$
 and $T_{5,3}.$  
 \end{example} 
 
 The importance of the upper left half is the following result which  captures the implications of the symmetry of the $n$-grids \cite[Corollary 9.6]{Hendel}. 
\begin{lemma} \label{lem:upperlefthalf}
With $n,m$ defined as in Definition \ref{def:upperlefthalf} all edge values in $T^{n,m}$ are determined by the edge values of the upper left half.
\end{lemma}

 \begin{definition}\label{def:ssubgrid}\cite[Equation (29)]{Hendel}
 With $n,m$ as in Definition \ref{def:upperlefthalf} we define for $s \ge 1$
 the $s$ subgrid of $T^{m,n}$ notationally indicated by $T^{n,m,s}$ to be the subgrid of $T^{m,n}$ with corners 
 $ T_{2s-1,s}, T_{n+1-s,s}, T_{n+1-s,n+2-2s}$   The  triangle-border or triangle boundary of $T^{m,n,s}$ notationally indicated $\partial T^{m,n,s}$ indicates
  the union of triangle sets, 
$ T_{r,s}, 2s-1 \le r \le n+1-s,
 T_{n+1-s,d}, s \le d \le n+2-2s,
 T_{r,(r-(s-1)}, 2s-1 \le r \le n+1-s.$
  (The $s$ subgrid  is also called the $s-rim.$)
 The edge-border or edge-boundary of the $T^{m,n,s}$  notationally indicated by
  $\partial^2 T^{n,m,s}$ indicates the union of edges
$ T_{r,s,L}, 2s-1 \le r \le n+1-s,
 T_{n+1-s,d,B}, s \le d \le n+2-2s,
 T_{r,(r-(s-1),R}, 2s-1 \le r \le n+1-s,$
  where throughout we have omitted the  superscript 
  $(m,n)$ for readability (that is, $T$ throughout refers to $T^{m,n}.)$
  
  The interior of $T^{m,n,s}$ notationally indicated by $Int(T^{m,n,s})$ is defined as the collection of edges in $T^{m,n,s}$ minus  the edges on the edge border of  $T^{m,n,s}$ (that is, $Int(T^{m,n,s}) = T^{m,n,s} - \partial^2 T^{m,n,s}.$)
 \end{definition}
 
\begin{figure}[ht!]
\begin{center}
\begin{tikzpicture}[scale = .6, line cap=round,line join=round,>=triangle 45,x=1.0cm,y=1.0cm,scale=.8]

\draw [line width=.8pt, color=blue] (-3.,10.577350269189628)--  (-4.,8.84529946162075);

\draw [line width=.8pt, color=blue] (-4.,8.84529946162075)--  (-5.,7.113248654051867);

\draw [line width=.8pt, color=blue] (-5.,7.113248654051867)--  (-6.,5.381197846482991);

\draw [line width=.8pt, color=blue] (-6.,5.381197846482991)-- (-7.,3.6491470389141107);

\draw [line width=.8pt, color=blue] (-7.,3.6491470389141107)--(-8.,1.9170962313452288)--(-9.,0.1850454237763482)-- (-10.,-1.5470053837925324);

\draw [line width=.8pt, color=blue] (-3.,10.577350269189628)-- (-2.,8.845299461620748);

\draw [line width=.8pt, color=blue] (-2.,8.845299461620748)--  (-1.,7.113248654051868);

\draw [line width=.8pt, color=blue] (-1.,7.113248654051868)-- (0.,5.38119784648299);

\draw [line width=.8pt, color=blue] (0.,5.38119784648299)-- (1.,3.6491470389141094);
\draw [line width=.8pt, color=blue] (1.,3.6491470389141094)-- (2.,1.9170962313452288)--(3.,0.1850454237763482) --(4.,-1.5470053837925324);
\draw [line width=.8pt, color=blue] (4.,-1.5470053837925324)-- (2.,-1.5470053837925324)-- (0.,-1.5470053837925324)-- (-2.,-1.5470053837925324)-- (-4.,-1.5470053837925324)-- (-6.,-1.5470053837925324)-- (-8.,-1.5470053837925324)-- (-10.,-1.5470053837925324);
\draw [line width=.8pt,dashed,color=blue] (1.,3.6491470389141094)-- (-1.,3.6491470389141134);
\draw [line width=.8pt,dashed, color=red] (-3.,3.6491470389141174)-- (-1.,3.6491470389141134);
\draw [line width=.8pt,dashed, color=red] (-3.,3.6491470389141174)--  (-5.,3.6491470389141125);
\draw [line width=.8pt,dashed, color=blue] (-5.,3.6491470389141125)-- (-7.,3.6491470389141107);

\draw [line width=.8pt,dashed,color=blue] (-4.,8.84529946162075)--(-2.,8.845299461620748)-- (-3.,7.113248654051869)--(-4.,8.84529946162075);
\draw [line width=.8pt, dashed,color=blue] (-3.,7.113248654051869)-- (-5.,7.113248654051867);
\draw [line width=.8pt, dashed,color=blue] (-4.,5.381197846482992)-- (-5.,7.113248654051867);
\draw [line width=.8pt,color=red]  (-4.,5.381197846482992)--(-3.,7.113248654051869);
\draw [line width=.8pt,color=red] (-3.,7.113248654051869)--  (-2.,5.381197846482992);
\draw [line width=.8pt, dashed,color=blue] (-1.,7.113248654051869)--  (-2.,5.381197846482992);
\draw [line width=.8pt, dashed,color=blue] (-1.,7.113248654051868)--(-3.,7.113248654051869);
\draw [line width=.8pt, dashed,color=blue] (-6.,5.381197846482991)--(-4.,5.381197846482992);
\draw [line width=.8pt, dashed,color=blue](-5.,3.6491470389141125)--(-6.,5.381197846482991);
\draw [line width=.8pt,color=red](-5.,3.6491470389141125)--(-4.,5.381197846482992);
\draw [line width=.8pt, dashed,color=red] (-4.,5.381197846482992)--  (-2.,5.381197846482992);
\draw [line width=.8pt, dashed,color=red] (-3.,3.6491470389141174)--(-4.,5.381197846482992);
\draw [line width=.8pt, dashed,color=red] (-3.,3.6491470389141174)--(-2.,5.381197846482992);
\draw [line width=.8pt, dashed,color=blue] (-2.,5.381197846482992)-- (0.,5.38119784648299)-- (-1.,3.6491470389141134);
\draw [line width=.8pt,color=red] (-1.,3.6491470389141134)--(-2.,5.381197846482992);
\draw [line width=.8pt,color=red] (-1.,3.6491470389141134)--(0.,1.9170962313452288)--(1.,0.1850454237763482);
\draw [line width=.8pt,color=red] (-5.,3.6491470389141134)--(-6.,1.9170962313452288)--(-7.,0.1850454237763482);
\draw [line width=.8pt, dashed,color=blue] (1.,3.6491470389141134)--(0.,1.9170962313452288)--(2.,1.9170962313452288)--(1.,0.1850454237763482);
\draw [line width=.8pt, dashed,color=blue] (-7.,3.6491470389141134)--(-6.,1.9170962313452288)--(-8.,1.9170962313452288)--(-7.,0.1850454237763482);
\draw [line width=.8pt, dashed,color=red] (-1.,3.6491470389141134)--(-2.,1.9170962313452288)--(0.,1.9170962313452288)--(-1.,0.1850454237763482)--(-2.,1.9170962313452288)--(-3.,0.1850454237763482)--(-4.,1.9170962313452288)--(-5.,0.1850454237763482)--(-6.,1.9170962313452288)--(-4.,1.9170962313452288)--(-5.,3.6491470389141134);
\draw [line width=.8pt,color=green] (-3.,3.6491470389141174)--(-2.,1.9170962313452288)-- (-4.,1.9170962313452288)-- (-3.,3.6491470389141174);
\draw [line width=.8pt, color=red] (1.,0.1850454237763482)-- (-1.,0.1850454237763482)-- (-3.,0.1850454237763482)-- (-5.,0.1850454237763482) -- (-7.,0.1850454237763482);
\draw [line width=.8pt, dashed,color=blue](1.,0.1850454237763482)--(3.,0.1850454237763482)-- (2.,-1.5470053837925324)--(1.,0.1850454237763482)-- (0.,-1.5470053837925324)-- (-1.,0.1850454237763482)-- (-2.,-1.5470053837925324)-- (-3.,0.1850454237763482)-- (-4.,-1.5470053837925324)-- (-5.,0.1850454237763482) -- (-6.,-1.5470053837925324)-- (-7.,0.1850454237763482)-- (-8.,-1.5470053837925324)--(-9.,0.1850454237763482) -- (-7.,0.1850454237763482);
%--
\begin{scriptsize}

\draw [fill=black] (-3.,10.577350269189628) circle (2.5pt);

\draw [fill=black] (-2.,5.381197846482992) circle (2.5pt);
 
\draw [fill=black] (-1.,7.113248654051868) circle (2.5pt);
 
\draw [fill=black] (-2.,8.845299461620748) circle (2.5pt);
 
\draw [fill=black] (-4.,8.84529946162075) circle (2.5pt);
 
\draw [fill=black] (-5.,7.113248654051867) circle (2.5pt);
 
\draw [fill=black] (-4.,5.381197846482992) circle (2.5pt);
 
\draw [fill=black] (-1.,3.6491470389141134) circle (2.5pt);
 
\draw [fill=black] (0.,5.38119784648299) circle (2.5pt);
 
\draw [fill=black] (1.,3.6491470389141094) circle (2.5pt);
 
\draw [fill=black] (-6.,5.381197846482991) circle (2.5pt);
 
\draw [fill=black] (-5.,3.6491470389141125) circle (2.5pt);
 
\draw [fill=black] (-7.,3.6491470389141107) circle (2.5pt);
 
\draw [fill=black] (-3.,3.6491470389141174) circle (2.5pt);

\draw [fill=black] (2.,1.9170962313452288) circle (2.5pt);

\draw [fill=black] (0.,1.9170962313452288) circle (2.5pt);

\draw [fill=black] (-2.,1.9170962313452288) circle (2.5pt);

\draw [fill=black] (-4.,1.9170962313452288) circle (2.5pt);

\draw [fill=black] (-6.,1.9170962313452288) circle (2.5pt);

\draw [fill=black] (-8.,1.9170962313452288) circle (2.5pt);

\draw [fill=black] (3.,0.1850454237763482) circle (2.5pt);

\draw [fill=black] (1.,0.1850454237763482) circle (2.5pt);

\draw [fill=black] (-1.,0.1850454237763482) circle (2.5pt);

\draw [fill=black] (-3.,0.1850454237763482) circle (2.5pt);

\draw [fill=black] (-5.,0.1850454237763482) circle (2.5pt);

\draw [fill=black] (-7.,0.1850454237763482) circle (2.5pt);

\draw [fill=black] (-9.,0.1850454237763482) circle (2.5pt);

\draw [fill=black] (4.,-1.5470053837925324) circle (2.5pt);

\draw [fill=black] (2.,-1.5470053837925324) circle (2.5pt);

\draw [fill=black] (0.,-1.5470053837925324) circle (2.5pt);

\draw [fill=black] (-2.,-1.5470053837925324) circle (2.5pt);

\draw [fill=black] (-4.,-1.5470053837925324) circle (2.5pt);

\draw [fill=black] (-6.,-1.5470053837925324) circle (2.5pt);

\draw [fill=black] (-8.,-1.5470053837925324) circle (2.5pt);

\draw [fill=black] (-10.,-1.5470053837925324) circle (2.5pt);

\end{scriptsize}

\end{tikzpicture}
 
\caption{An illustration of the triangular rims.  The edge boundaries of these rims are colored as follows: $T^{n,4,1}$ is blue, $T^{n,4,2}$ is red, $T^{n,4,3}$ is green, where $n$ is the number of rows in the initial $n$-grid $T^n$ all of whose edges are labeled 1 and which was reduced $n-7$ times. The boundaries of each rim are solid, and the interior edges are dashed. }
\label{fig:7gridrims}
\end{center}
\end{figure}
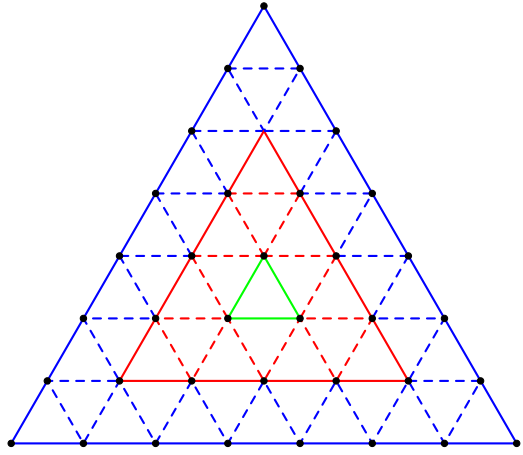

  \begin{example}\label{exa:7gridrims}
 Consider the 7 grid shown in Figure \ref{fig:7gridrims}. The 1-rim has corners $T_{1,1}, T_{7,1}, T_{7,7}$
 $\langle 1,1 \rangle,\langle 7,1 \rangle, \langle 7,7 \rangle$ and is colored blue; the 2-rim has corners
 $T_{3,2}, T_{6,2}, T_{6,5}$
  and is colored red; the three rim is the singleton green triangle $T_{5,3}.$  
 \end{example}

\section{A Conjecture}\label{sec:motivation}
 
In an attempt to understand the reduction algorithm, the authors began  to look at the ``initial behavior'' of the edge resistances as an $n$-grid undergoes the stages of the reduction algorithm.  The authors noticed that with each transformation, the location of edges with resistance equal to one decreased in a predictable pattern,  Moreover, it was noticed that the non-one edges also maintained a specific pattern dependent on their location in a rim.  We formalize these observations in the following conjecture.

\begin{conjecture}[Vanishing Ones Conjecture]\label{con:main} 
For integer $n \ge 1,$ we have the following:\\
(a)  $T^{n,n,1}=1.$\\
(b) For $1 \le s \le \lfloor \frac{n+1}{4} \rfloor,$  $Interior(T^{n,n-s,s})$ is equal to one. \\
(c) With $s$ as in (b), for an edge in the complement of the interior of an $s$-subgrid: its edge value is strictly less than one if it lies on the edge boundary of some $s'$-grid, $1 \le s' < s$; its edge value is strictly greater than one otherwise.\\
(d) For any $s,$ $\lfloor \frac{n+1}{4} \rfloor < s \le n-1$ there are no edges with label 1 in 
$T^{n,n-s}$ (the ones ``vanish").\\
\end{conjecture}

We have numerically verified this conjecture on all $n$-grids, $1 \le n \le 250.$ Conjecture \ref{con:main}(a) is true by assumption and was included for purposes of completeness. Conjecture \ref{con:main}(b) is proven at Corollary \ref{cor:provenlater}.    

However, despite the clear evidence and simplicity of the statement, the proof of the remaining parts of Conjecture \ref{con:main} remains elusive.   A special cases of (c) is proven at Corollary \ref{cor:specialcase}.   
 
\section{Proof Methods.}\label{sec:proofmethods}

All results in this paper are proven by applying Algorithm \ref{alg:supertri} which in turn involves computing series,  $\Delta$--Y, and Y--$\Delta$ transformations.  We will follow the technique used by Hendel \cite{Hendel} who approached each proof using the five steps (panels) A-E presented in Figure \ref{fig:5panels}. In other words, each proof will consist of (Step A) identification of the triangles and edge values in some $n$-grid used in the computation, (Steps B, C) application of the $\Delta-Y$ transforms and ignoring any tails, (Step D) performing any relevant series computations (if border edges are involved), and (Step E) performing $Y-\Delta$ transformations resulting in the edge value of a triangle in the reduced $(n-1)$-grid. The proofs are typically summarized in a figure sequentially showing Steps A-E.

Unlike \cite{Hendel} where a typical proof applied to all edges of a triangle, proofs in this paper will very often focus on specific edges.   There are only 4 cases that have to be considered  to develop a comprehensive suite of lemmas that cover all needed proofs. These four cases are presented in this section. Each lemma is given a mnemonical name to facilitate reference later in the paper.

\begin{lemma}  \label{lem:4triangles}[Base Edge]
For given integers $n \ge 3, r \le n-2, 1 \le d,$ the edge-value of  
$T^{n,n-1}_{r,d,B}$
    is computed from the 9 edge-values  in the triangles $T^{n,n}_{r+1,d},\;T^{n,n}_{r+1,d}$ and $T^{n,n}_{r+2,d+1}$ (Figure ~\ref{fig:4triangles} presents a general case).
\end{lemma}
\begin{proof} Algorithm \ref{alg:supertri}. To clarify the proof, we note that the edges labeled a-j in Panel A of Figure \ref{fig:4triangles} identify the 9 edges of the parent grid needed to compute the target base edge of the child grid. This identification is in fact what is required by Step A of Algorithm~\ref{alg:supertri}. Then Steps B and C require  discarding tails and performing  $\Delta-Y$ transformations as   shown in Panels B and C. Finally, panels D and E illustrate the required computations of Steps D and E, the computation of $Y-\Delta$ transformations. The function derived is all that is needed to perform the computations in future sections. 
\end{proof}

Notice that the base edge is a function of 9 edge values as indicated in Figure \ref{fig:4triangles}. The following notation will be used to summarize the relationship.
\begin{equation}\label{equ:Fbaseedge}
T^{n,n-1}_{r,d,B} = F(T^{n,n}_{r+1,d}, T^{n,n}_{r+1,d+1},
T^{n,n}_{r+2,d+1}),
\end{equation}
where $F$ indicates some function, and each triangle argument is understood to represent the 3 edge resistances of that triangle (say in clockwise order starting with the left edge). While the notation is primarily  mnemonical, it is compact and very useful in summarizing functional dependencies in proofs.

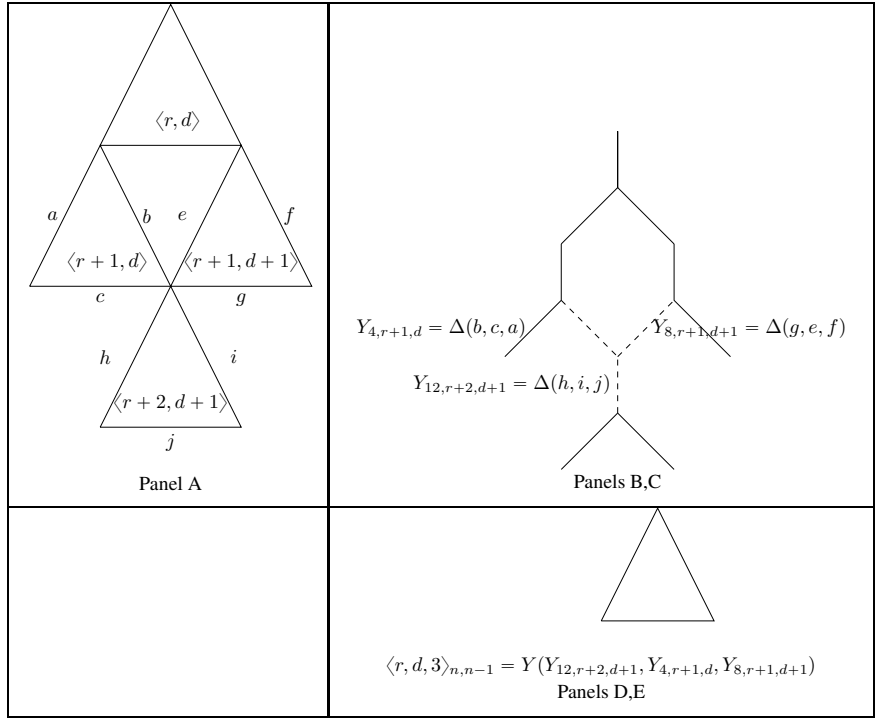
\begin{figure}[ht]
\begin{center}
\begin{tabular} {|c|c|}
\hline
\scalebox{.75}{
\begin{tikzpicture}[xscale=1.25, yscale=1.25]

\draw  (1,0)--(2,2)--(3,0)--(1,0);
\draw   (0,2)--(1,4)--(2,2)--(0,2);
\draw   (2,2)--(3,4)--(4,2)--(2,2);
\draw   (1,4)--(2,6)--(3,4)--(1,4);

\node [left] at (1.25,1) {$h$};
\node [right] at (2.75,1) {$i$};
\node [below] at (2,0) {$j$};

\node [above] at (2.0,.1) {$\langle r+2,d+1 \rangle$};
\node [above] at (1.1,2.1) {$\langle r+1,d \rangle$};
\node [above] at (3.0,2.1) {$\langle r+1,d+1 \rangle$};
\node [above] at (2.1,4.1) {$\langle r,d \rangle$};

\node [left] at (.5,3) {$a$};
\node [below] at (1,2) {$c$};
\node [right] at (1.5 ,3) {$b$};

\node [right] at (2 ,3) {$e$};
\node [right] at (3.5,3)  {$f$};
\node [below] at (3,2) {$g$};

\node [above] at (2,-1) {Panel A};
\end{tikzpicture}}
&

\scalebox{.75}{
\begin{tikzpicture} [xscale=1,yscale=1]
\draw (1,0)--(2,1);
\draw [dashed] (2,1)--(2,2);
\draw (2,1)--(3,0);
\draw (0,2)--(1,3)--(1,4);
\draw [dashed] (1,3)--(2,2);
\draw [dashed]  (2,2)--(3,3);
\draw (3,4)--(3,3)--(4,2);
\draw  (1,4)--(2,5)--(2,6)--(2,5)--(3,4);

 \node [left] at (2,1.5)  
{$Y_{12,r+2,d+1}=\Delta(h,i,j)$};
\node [left] at (.5,2.5) {$Y_{4,r+1,d}=\Delta(b,c,a) $};
\node [right] at (2.5,2.5) {$Y_{8,r+1,d+1} = \Delta(g,e,f)$};

\node [above] at (2,-.5) {Panels B,C};
 \end{tikzpicture}
 }

\\
\hline &

\scalebox{.75}{
\begin{tikzpicture} 
\draw (0,0)--(1,2)--(2,0)--(0,0);

\node [below] at (0,-0.5)  
{$
\langle r,d,3 \rangle_{n,n-1} = Y(Y_{12,r+2,d+1},Y_{4,r+1,d},Y_{8,r+1,d+1})$};

\node [below] at (0,-1) {Panels D,E};

\end{tikzpicture}
}
\\
\hline
\end{tabular}
\end{center}
\caption{
(Base Edge Lemma) Panels A-E for calculation of the base edge value used in Lemma \ref{lem:4triangles}. In Panel A, $a-i$ are edge values with row and diagonal coordinates of each triangle  indicated  in the triangle's interior. The functions $Y$ and $\Delta$ are defined in Definitions \ref{def:dy}-\ref{def:yd}. For typographical reasons, here and throughout other figures presented in this paper,  angle bracket notation is used to indicate specific triangles, i..e.
$\langle r, d \rangle = T_{r,d}.$ The notation for $Y$ legs was provided after the definition of Definition 2.7.
}\label{fig:4triangles}
\end{figure}

\begin{lemma} \label{lem:2triangles}[Boundary Edge]
For given integers $n \ge 2, r \le n-1,$ the edge value of 
$T^{n,n-1}_{r,1,L}$
 is computed from the 6 edge-values in the triangles $T^{n,n}_{r,1}$ and $T^{n,n}_{r+1,1}$ (Figure \ref{fig:2triangles} presents a general case).
\end{lemma}
\begin{proof} Algorithm \ref{alg:supertri} 
\end{proof}

The associated functional equation is,
\begin{equation}\label{equ:Fboundaryedge}
T^{n,n-1}_{r,1,L} = F(T^{n,n}_{r,1}, T^{n,n}_{r+1,1}).
\end{equation}

\begin{figure}[ht!]
\begin{center}
\begin{tabular} {|c|c|}
\hline
\scalebox{.75}{

\begin{tikzpicture}[xscale=1,yscale=1]

\draw   (3,2)--(4,4)--(5,2)--(3,2);
\draw    (4,4)--(5,6)--(6,4)--(4,4);

\node [above] at (5,4) {$\langle r,1 \rangle$};
\node [above] at (4,2) {$\langle r+1,1 \rangle$};

\node [left] at (4,5) {$a$};
\node [left] at (3.25,3) {$e$};

\node [right] at (5.5,5)
{$b$};  
\node [right] at (4.5,3)
{$f$};
 
\node [below,right] at (4,1.75) 
{$g$};
\node [below,right] at (5,3.75) 
{$c$};

\node [below] at (4,1.5) {Panel A};

\end{tikzpicture}
}
&
\scalebox{.75}{
\begin{tikzpicture}[xscale=1, yscale=1]
 
\draw  (1,2)--(2,3)--(3,2);
\draw [dashed] (2,3)--(2,4);
\draw [dashed] 
(2,4)--(3,5);
\draw (4,4)--(3,5)--(3,6);

\node [right] at (0,3.5) 
{$Y_{12,r+1,1}=   
\Delta(e,f,g)$};
\node [right] at (0,4.5) 
{$Y_{8,r,1}= 
\Delta(c,a,b)$};

\node [below] at (6,3) {Panels B, C};
\end{tikzpicture}
}
\\
\hline
\qquad
&  
\scalebox{.75}{
\begin{tikzpicture}[xscale=1,yscale=1]
 
\draw (3,4)--(4,6)--(5,4)--(3,4);

\node [above] at (4.1,4) {$\langle r,1 \rangle_{n,n-1}$};

\node [left] at (3,5)
{$\langle r,1,1 \rangle_{n,n-1}=Y_{8,r,1}+Y_{12,r+1,1} $};
 
\node [below] at (2,4.5) {Panel D};

\end{tikzpicture}
}
\\
\hline

\end{tabular}
\end{center}
\caption{(Boundary Edge Lemma) Panels A-D for calculation of the left side edge value on the boundary of the $n$-grid  used in Lemma \ref{lem:2triangles}.} 
\label{fig:2triangles}
\end{figure}
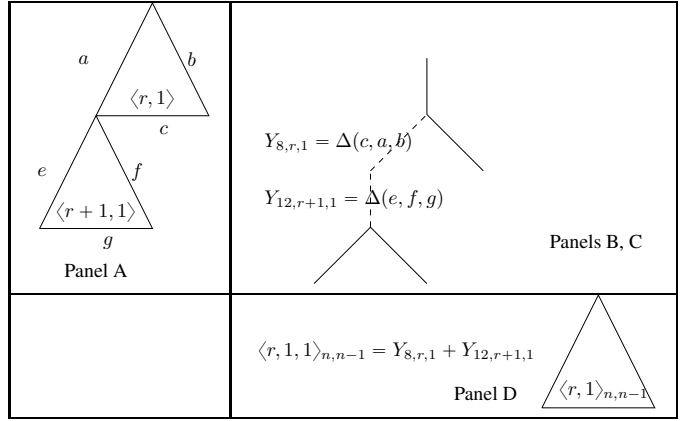

\begin{lemma} \label{lem:3trianglesleft}[Left Edge]
For given integers $n \ge 2, r \le n-1, d \ge 2$, the edge-value of 
$T^{n,n-1}_{r,d,L}$
 is computed from the 9 edge-values in the triangles $T^{n,n}_{r,d-1,}, T^{n,n}_{r,d}$ and, $ T^{n,n}_{r+1,d}$ (Figure \ref{fig:3trianglesleft} presents a general case).
\end{lemma}

\begin{proof} Algorithm \ref{alg:supertri} 
\end{proof}
The associated functional equation is
\begin{equation}\label{equ:Fleftedge}
T^{n,n-1}_{r,d,L} = F(T^{n,n}_{r,d-1,}, T^{n,n}_{r,d}, T^{n,n}_{r+1,d})
\end{equation}

\begin{figure}[ht!]
\begin{center}
\begin{tabular} {|c|}
\hline
\scalebox{.75}{
\begin{tikzpicture}[xscale=1,yscale=1]

\draw  (2,4)--(3,6)--(4,4)--(2,4);
\draw   (4,4)--(5,6)--(6,4)--(4,4);
\draw   (3,2)--(4,4)--(5,2)--(3,2);

\node  at (3,4.25) {$\langle r, d-1 \rangle$};
\node  at (5,4.25) {$\langle r, d \rangle$};
\node  at (4,2.25) {$\langle r+1, d\rangle$};

\node [right] at (2,5.5) {$a$};
\node [right] at (3.5,5.5) {$b$};
\node [below] at (3,4) {$c$};

\node [left] at (4.5,5) {$e$};
\node [right] at (5.5,5) {$f$};
\node [below] at (5.5,4) {$g$};

\node [left] at (3.5,3) {$h$};
\node [right] at (4.5,3) {$i$};
\node [below] at (4,2) {$j$};

\node [below] at (3.5,1.5) {Panel A};
\end{tikzpicture}
}
\\

\hline
\scalebox{.75}{

\begin{tikzpicture}[xscale=1,yscale=1]

\draw  (4.5,4)--(5.5,5)--(5.5,6);
\draw [blue] (5.5,5)--(6.5,4);
\draw  (7.5,6)--(7.5,5)--(8.5,4);
\draw [red] (6.5,4)--(7.5,5);
\draw   (5.5,2)--(6.5,3)--(7.5,2);
\draw [dashed](6.5,4)--(6.5,3);

\node [right, blue] at (2.5,4.5) 
{$Y_{4,r,d-1}= \Delta(b,c,a)$};
\node [right,red] at (8.5,4.5) 
{$Y_{8,r,d}= \Delta(g,e,f) $};
\node [right] at (4.5,3.5) 
{$Y_{12,r+1,d}=\Delta(h,i,j)$};

\node [below] at (5,1.5) {Panel B, C};
 
\end{tikzpicture}
}

\\
 \hline
\scalebox{.75}{
\begin{tikzpicture}[xscale=1,yscale=1]

\draw  (1,2)--(4,6)--(7,2)--(1,2);
\node [right] at (1,4) {$L=Y(Y_{4,r,d-1}, Y_{8,r,d}, Y_{12,r+1,d})$};
 
\node [above] at (4,2) {$\langle r,d \rangle_{n,n-1}$};

\node [below] at (4,1.5) {Panel D-E};

\end{tikzpicture}
}
\\  
\hline
\end{tabular}
\caption{(Left Edge Lemma) Panels A-E for calculation of  the left edge value used in Lemma \ref{lem:3trianglesleft}. To increase readability, the labels for the 8 O'clock and 4 O'clock legs are color coded identically with their labels.}\label{fig:3trianglesleft}
\end{center}
\end{figure}
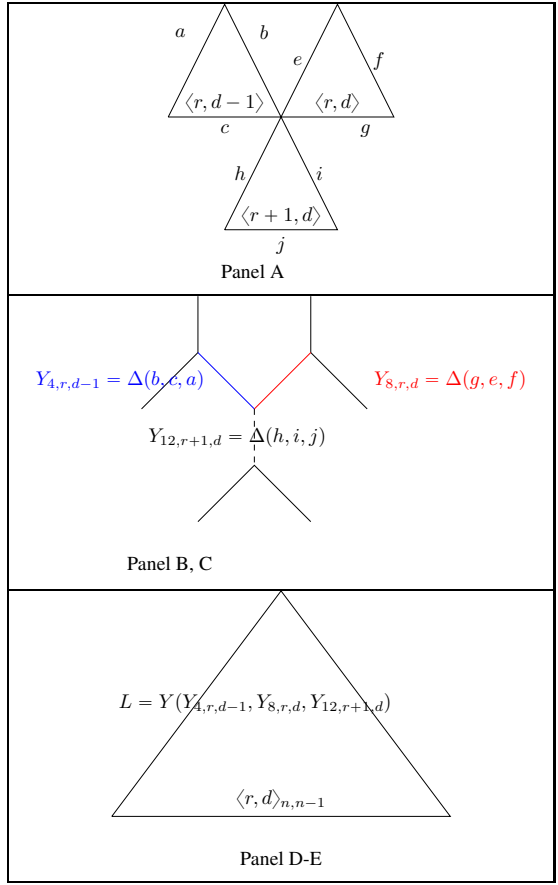

\begin{lemma} \label{lem:3trianglesright}[Right Edge]
Given integers $n \ge 2, r \le n-1, d \le r-1$ to compute 
$T^{n,n-1}_{r,d,R}$
it suffices to use the triangles, edge values, and functions presented in Figure \ref{fig:3trianglesright}.
\end{lemma}
\begin{proof} Algorithm \ref{alg:supertri} 
\end{proof}
The associated functional equation is
\begin{equation}\label{equ:Frightedge}
T^{n,n-1}_{r,d,R} = F(T^{n,n}_{r,d}, T^{n,n}_{r+1,d},T^{n,n}_{r+1,d+1}).
\end{equation}

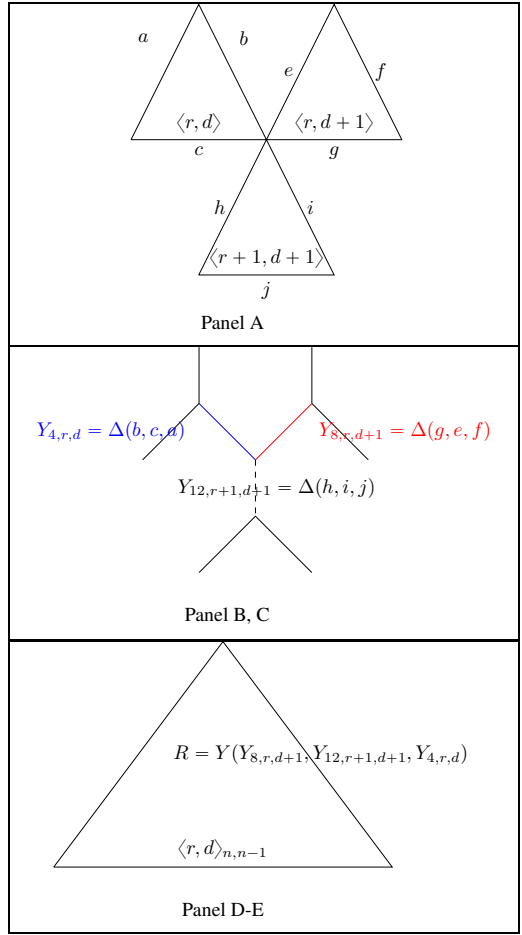
\begin{figure}[ht!]
\begin{center}
\begin{tabular} {|c|}
\hline
\scalebox{.75}{
\begin{tikzpicture}[xscale=1.2,yscale=1.2]

\draw  (2,4)--(3,6)--(4,4)--(2,4);
\draw   (4,4)--(5,6)--(6,4)--(4,4);
\draw   (3,2)--(4,4)--(5,2)--(3,2);

\node  at (3,4.25) {$\langle r, d \rangle$};
\node  at (5,4.25) {$\langle r, d+1  \rangle$};
\node  at (4,2.25) {$\langle r+1, d+1 \rangle$};

\node [right] at (2,5.5) {$a$};
\node [right] at (3.5,5.5) {$b$};
\node [below] at (3,4) {$c$};

\node [left] at (4.5,5) {$e$};
\node [right] at (5.5,5) {$f$};
\node [below] at (5,4) {$g$};

\node [left] at (3.5,3) {$h$};
\node [right] at (4.5,3) {$i$};
\node [below] at (4,2) {$j$};

\node [below] at (3.5,1.5) {Panel A};
\end{tikzpicture}}
\\

\hline
\scalebox{.75}{
\begin{tikzpicture}[xscale=1,yscale=1]

\draw  (3.5,4)--(4.5,5)--(4.5,6);
\draw [blue] (4.5,5)--(5.5,4);
\draw  (6.5,6)--(6.5,5)--(7.5,4);
\draw [red] (5.5,4)--(6.5,5);
\draw   (4.5,2)--(5.5,3)--(6.5,2);
\draw [dashed](5.5,4)--(5.5,3);

\node [right, blue] at (1.5,4.5) 
{$Y_{4,r,d}= \Delta(b,c,a)$};
\node [right,red] at (6.5,4.5) 
{$Y_{8,r,d+1}= \Delta(g,e,f) $}; 
\node [right] at (4,3.5) 
{$Y_{12,r+1,d+1}=\Delta(h,i,j)$};

\node [below] at (5,1.5) {Panel B, C};
 
\end{tikzpicture}
}
\\
 \hline
\scalebox{.75}{
\begin{tikzpicture}[xscale=1,yscale=1]

\draw  (1,2)--(4,6)--(7,2)--(1,2);
\node [right] at (3,4) {$R=Y(Y_{8,r,d+1}, Y_{12,r+1,d+1}, Y_{4,r,d})$};
 
\node [above] at (4,2) {$\langle r,d \rangle_{n,n-1}$};

\node [below] at (4,1.5) {Panel D-E};

\end{tikzpicture}

}
\\  
\hline
\end{tabular}
\caption{(Right Edge Lemma) Panels A-E for calculation of an right edge resistance used in Lemma \ref{lem:3trianglesright}}\label{fig:3trianglesright}
\end{center}
\end{figure}

Equations \eqref{equ:Fbaseedge} - \eqref{equ:Frightedge} imply the heuristic that the edge values of the triangle $T^{n,n-(i+1)}_{r,d},\; d \neq 1$  
is a function of triangles in the parent grid of the form $T^{n,n-i}_{r+u,r+v}, u \in \{0,1,2\}, 
v \in \{-1,0,1\}.$ This immediately yields some interesting results presented in the next section.

\section{Illustrations and Easy Consequences of the Proof Methods}\label{sec:illandconseq}

This section begins by presenting two straightforward consequences of the proof methods of Section \ref{sec:proofmethods}. We first present a corollary illustrating how the figures associated with the lemmas in Section \ref{sec:proofmethods} naturally provide explicit computations sufficient to calculate all edge values in a single reduction.

\begin{corollary}\label{cor:tnnminus1istwothirds}
Let $n \ge 2.$ Then $\partial^2 T^{n,n-1} = \frac{2}{3},$ and 
$Int(T^{n,n-1})=1.$ In words, the edge boundary of $T^{n,n-1}$ is identically $\frac{2}{3}$ while all other edges are 1.
\end{corollary}
\begin{proof} To acclimate the reader to the proof methods of Section \ref{sec:proofmethods} we present two proofs.

First, we prove this result by appealing directly to the reduction algorithm and calculations. By Algorithm~\ref{alg:supertri}, using a $\Delta-Y$ transformation, we start the proof by converting the $n$ rows of $T^n$ to $n$ rows of $Y$s. 

Since $\Delta(1,1,1)=\frac{1}{3},$ all $Y$-edges have weight
$\frac{1}{3}$.  By Algorithm~\ref{alg:supertri} there are now two cases to consider. Along the boundary, the edges of $T^{n,n-1}$ arise from a series transformation and hence have weight $\frac{1}{3}+\frac{1}{3}=\frac{2}{3}.$  All remaining edges in $T^{n,n-1}$  arise from $Y-\Delta$ transformations and hence have 
weight 
$Y(\frac{1}{3},\frac{1}{3},\frac{1}{3}) = 1.$

The second proof reflects the fact that the algorithmic descriptions have already been provided in the figures accompanying the proofs of Section \ref{sec:proofmethods}. First consider the left-boundary edges of $T^{n,n-1}.$  By the Boundary Edge Lemma,  
$\langle r,1, 1 \rangle_{n,n-1}
= \Delta(\frac{1}{3}, \frac{1}{3}, \frac{1}{3}) + \Delta(\frac{1}{3}, \frac{1}{3}, \frac{1}{3}) = \frac{2}{3}$. For the remaining edges of $T^{n,n-1}$ by the Left, Right, and Base Edge Lemmas and accompanying figures, the computation of the edge values arises from a $Y$ function applied to three arguments each of which is a $\Delta$ function whose three arguments are identically one. The proof is therefore completed by noting that $\Delta(1,1,1)=\frac{1}{3}$ and
$Y(\frac{1}{3},\frac{1}{3},\frac{1}{3})=1.$
\end{proof}

\begin{remark} In the sequel, we will suffice with proofs based on the appropriate lemmas and accompanying figures of Section \ref{sec:proofmethods}.   
\end{remark}

\begin{lemma}\label{lem:interior} For given integers $n \geq 3$ and $ 1 \leq s \leq \lfloor n/4 \rfloor$,  we have \begin{enumerate}
\item[(i)] $Int(T^{n,n-s,s})=1$, 
\item[(ii)] $\partial^2 T^{n,n-s,s}=c(s),$ where
$c(s)$ is some constant only dependent on $s$.
\end{enumerate}

\end{lemma}
\begin{proof}
	Before beginning the proof  we first define $T^p=T^{n,n-s,s}$ and $T^c = T^{n,n-s-1,s+1}$ where $p$ and $c$ stand for parent and child grid. The next statement, needed in the formal statement of the induction assumption, follows from the definition of subgrid (Definition~\ref{def:ssubgrid}) and the lemmas in Section~\ref{sec:proofmethods}.\\
 
(iii)  For edge $e  \in T^c (Int(T^c))$ the edge values used as arguments to compute $e$ all belong to $T^p(Int(T^p)).$\\
  	
With these preliminaries we now present the proof of the lemma. The proof is  by induction on $s$.  The base case, $s=1$  is given by Corollary~\ref{cor:tnnminus1istwothirds}. We explicitly note that (i)-(iii) hold for this case.

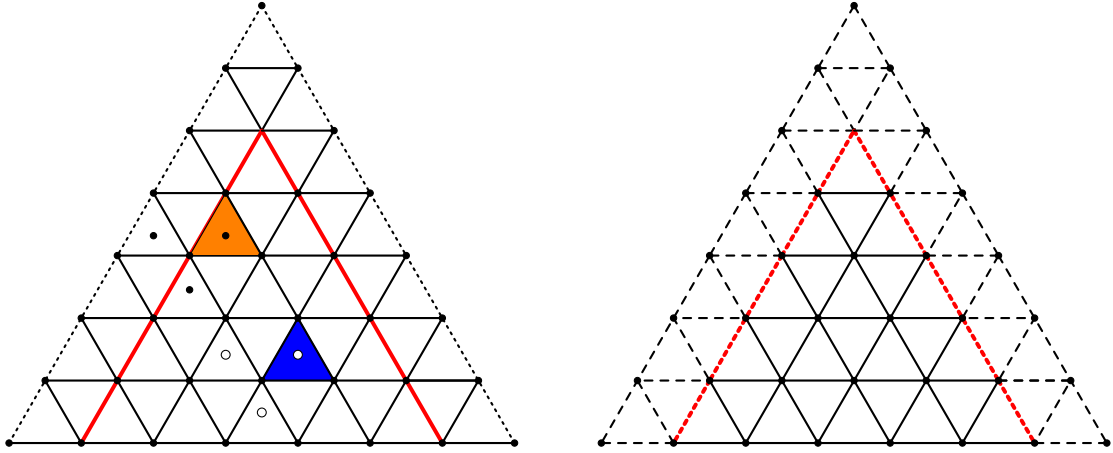
\begin{figure}[ht!]
\begin{center}

\begin{tikzpicture}[scale = .6, line cap=round,line join=round,>=triangle 45,x=1.0cm,y=1.0cm,scale=.8]

\draw [line width=.8pt, dotted, color=black] (-3.,10.577350269189628)--  (-4.,8.84529946162075);

\draw [line width=.8pt, dotted, color=black] (-4.,8.84529946162075)--  (-5.,7.113248654051867);

\draw [line width=.8pt, dotted, color=black] (-5.,7.113248654051867)--  (-6.,5.381197846482991);

\draw [line width=.8pt, dotted, color=black] (-6.,5.381197846482991)-- (-7.,3.6491470389141107);

\draw [line width=.8pt, dotted, color=black] (-7.,3.6491470389141107)--(-8.,1.9170962313452288)--(-9.,0.1850454237763482)-- (-10.,-1.5470053837925324);

\draw [line width=.8pt, dotted, color=black] (-3.,10.577350269189628)-- (-2.,8.845299461620748);

\draw [line width=.8pt, dotted, color=black] (-2.,8.845299461620748)--  (-1.,7.113248654051868);

\draw [line width=.8pt, dotted, color=black] (-1.,7.113248654051868)-- (0.,5.38119784648299);

\draw [line width=.8pt, dotted, color=black] (0.,5.38119784648299)-- (1.,3.6491470389141094);
\draw [line width=.8pt, dotted, color=black] (1.,3.6491470389141094)-- (2.,1.9170962313452288)--(3.,0.1850454237763482) --(4.,-1.5470053837925324);
\draw [line width=.8pt,  color=black] (4.,-1.5470053837925324)-- (2.,-1.5470053837925324)-- (0.,-1.5470053837925324)-- (-2.,-1.5470053837925324)-- (-4.,-1.5470053837925324)-- (-6.,-1.5470053837925324)-- (-8.,-1.5470053837925324)-- (-10.,-1.5470053837925324);
\draw [line width=.8pt,color=black] (1.,3.6491470389141094)-- (-1.,3.6491470389141134);
\draw [line width=.8pt, color=black] (-3.,3.6491470389141174)-- (-1.,3.6491470389141134);
\draw [line width=.8pt, color=black] (-3.,3.6491470389141174)--  (-5.,3.6491470389141125);
\draw [line width=.8pt, color=black] (-5.,3.6491470389141125)-- (-7.,3.6491470389141107);

\draw [line width=.8pt,color=black] (-4.,8.84529946162075)--(-2.,8.845299461620748)-- (-3.,7.113248654051869)--(-4.,8.84529946162075);
\draw [line width=.8pt, color=black] (-3.,7.113248654051869)-- (-5.,7.113248654051867);
\draw [line width=.8pt, color=black] (-4.,5.381197846482992)-- (-5.,7.113248654051867);
\draw [line width=1.5pt,color=red]  (-4.,5.381197846482992)--(-3.,7.113248654051869);
\draw [line width=1.5pt,color=red] (-3.,7.113248654051869)--  (-2.,5.381197846482992);
\draw [line width=.8pt, color=black] (-1.,7.113248654051869)--  (-2.,5.381197846482992);
\draw [line width=.8pt, color=black] (-1.,7.113248654051868)--(-3.,7.113248654051869);
\draw [line width=.8pt, color=black] (-6.,5.381197846482991)--(-4.,5.381197846482992);
\draw [line width=.8pt, color=black](-5.,3.6491470389141125)--(-6.,5.381197846482991);
\draw [line width=1.5pt,color=red](-5.,3.6491470389141125)--(-4.,5.381197846482992);
\draw[fill=orange] (-5.,3.6491470389141125)--(-4.,5.381197846482992) -- (-3.,3.6491470389141125) -- cycle;

\draw [fill=black] (-4.,4.2) circle (2.5pt);
\draw [fill=black] (-6.,4.2) circle (2.5pt);
\draw [fill=black] (-5.,2.7) circle (2.5pt);

\draw[fill=blue] (-3.,0.1850454237763482)--(-2.,1.9170962313452288) -- (-1.,0.1850454237763482) -- cycle;

\draw [fill=white] (-2.,0.9) circle (3.5pt);
\draw [fill=white] (-4.,0.9) circle (3.5pt);
\draw [fill=white] (-3.,-.7) circle (3.5pt);

\draw [line width=.8pt, color=black] (-4.,5.381197846482992)--  (-2.,5.381197846482992);
\draw [line width=.8pt, color=black] (-3.,3.6491470389141174)--(-4.,5.381197846482992);
\draw [line width=.8pt, color=black] (-3.,3.6491470389141174)--(-2.,5.381197846482992);
\draw [line width=.8pt, color=black] (-2.,5.381197846482992)-- (0.,5.38119784648299)-- (-1.,3.6491470389141134);
\draw [line width=1.5pt,color=red] (-1.,3.6491470389141134)--(-2.,5.381197846482992);
\draw [line width=1.5pt,color=red] (-1.,3.6491470389141134)--(0.,1.9170962313452288)--(1.,0.1850454237763482);
\draw [line width=1.5pt,color=red] (-5.,3.6491470389141134)--(-6.,1.9170962313452288)--(-7.,0.1850454237763482);
\draw [line width=.8pt, color=black] (1.,3.6491470389141134)--(0.,1.9170962313452288)--(2.,1.9170962313452288)--(1.,0.1850454237763482);
\draw [line width=.8pt, color=black] (-7.,3.6491470389141134)--(-6.,1.9170962313452288)--(-8.,1.9170962313452288)--(-7.,0.1850454237763482);
\draw [line width=.8pt, color=black] (-1.,3.6491470389141134)--(-2.,1.9170962313452288)--(0.,1.9170962313452288)--(-1.,0.1850454237763482)--(-2.,1.9170962313452288)--(-3.,0.1850454237763482)--(-4.,1.9170962313452288)--(-5.,0.1850454237763482)--(-6.,1.9170962313452288)--(-4.,1.9170962313452288)--(-5.,3.6491470389141134);
\draw [line width=.8pt,color=black] (-3.,3.6491470389141174)--(-2.,1.9170962313452288)-- (-4.,1.9170962313452288)-- (-3.,3.6491470389141174);
\draw [line width=.8pt, color=black] (1.,0.1850454237763482)-- (-1.,0.1850454237763482)-- (-3.,0.1850454237763482)-- (-5.,0.1850454237763482) -- (-7.,0.1850454237763482);
\draw [line width=.8pt, color=black](1.,0.1850454237763482)--(3.,0.1850454237763482)-- 
(2.,-1.5470053837925324);
\draw [line width=1.5pt, color=red](1.,0.1850454237763482)--(2.,-1.5470053837925324);
\draw [line width=.8pt, color=black](1.,0.1850454237763482)--(3.,0.1850454237763482);
\draw [line width=.8pt, color=black](1.,0.1850454237763482)-- (0.,-1.5470053837925324)-- (-1.,0.1850454237763482)-- (-2.,-1.5470053837925324)-- (-3.,0.1850454237763482)-- (-4.,-1.5470053837925324)-- (-5.,0.1850454237763482) -- (-6.,-1.5470053837925324)-- (-7.,0.1850454237763482);
\draw [line width=.8pt, color=black](-8.,-1.5470053837925324)--(-9.,0.1850454237763482)-- (-7.,0.1850454237763482);
\draw [line width=1.5pt, color=red](-7.,0.1850454237763482)-- (-8.,-1.5470053837925324);%--
\begin{scriptsize}

\draw [fill=black] (-3.,10.577350269189628) circle (2.5pt);

\draw [fill=black] (-2.,5.381197846482992) circle (2.5pt);
 
\draw [fill=black] (-1.,7.113248654051868) circle (2.5pt);
 
\draw [fill=black] (-2.,8.845299461620748) circle (2.5pt);
 
\draw [fill=black] (-4.,8.84529946162075) circle (2.5pt);
 
\draw [fill=black] (-5.,7.113248654051867) circle (2.5pt);
 
\draw [fill=black] (-4.,5.381197846482992) circle (2.5pt);
 
\draw [fill=black] (-1.,3.6491470389141134) circle (2.5pt);
 
\draw [fill=black] (0.,5.38119784648299) circle (2.5pt);
 
\draw [fill=black] (1.,3.6491470389141094) circle (2.5pt);
 
\draw [fill=black] (-6.,5.381197846482991) circle (2.5pt);
 
\draw [fill=black] (-5.,3.6491470389141125) circle (2.5pt);
 
\draw [fill=black] (-7.,3.6491470389141107) circle (2.5pt);
 
\draw [fill=black] (-3.,3.6491470389141174) circle (2.5pt);

\draw [fill=black] (2.,1.9170962313452288) circle (2.5pt);

\draw [fill=black] (0.,1.9170962313452288) circle (2.5pt);

\draw [fill=black] (-2.,1.9170962313452288) circle (2.5pt);

\draw [fill=black] (-4.,1.9170962313452288) circle (2.5pt);

\draw [fill=black] (-6.,1.9170962313452288) circle (2.5pt);

\draw [fill=black] (-8.,1.9170962313452288) circle (2.5pt);

\draw [fill=black] (3.,0.1850454237763482) circle (2.5pt);

\draw [fill=black] (1.,0.1850454237763482) circle (2.5pt);

\draw [fill=black] (-1.,0.1850454237763482) circle (2.5pt);

\draw [fill=black] (-3.,0.1850454237763482) circle (2.5pt);

\draw [fill=black] (-5.,0.1850454237763482) circle (2.5pt);

\draw [fill=black] (-7.,0.1850454237763482) circle (2.5pt);

\draw [fill=black] (-9.,0.1850454237763482) circle (2.5pt);

\draw [fill=black] (4.,-1.5470053837925324) circle (2.5pt);

\draw [fill=black] (2.,-1.5470053837925324) circle (2.5pt);

\draw [fill=black] (0.,-1.5470053837925324) circle (2.5pt);

\draw [fill=black] (-2.,-1.5470053837925324) circle (2.5pt);

\draw [fill=black] (-4.,-1.5470053837925324) circle (2.5pt);

\draw [fill=black] (-6.,-1.5470053837925324) circle (2.5pt);

\draw [fill=black] (-8.,-1.5470053837925324) circle (2.5pt);

\draw [fill=black] (-10.,-1.5470053837925324) circle (2.5pt);
\end{scriptsize}

\end{tikzpicture}\quad\quad\quad
\begin{tikzpicture}[scale = .6, line cap=round,line join=round,>=triangle 45,x=1.0cm,y=1.0cm,scale=.8]

\draw [line width=.8pt, dashed, color=black] (-3.,10.577350269189628)--  (-4.,8.84529946162075);

\draw [line width=.8pt, dashed, color=black] (-4.,8.84529946162075)--  (-5.,7.113248654051867);

\draw [line width=.8pt, dashed, color=black] (-5.,7.113248654051867)--  (-6.,5.381197846482991);

\draw [line width=.8pt, dashed, color=black] (-6.,5.381197846482991)-- (-7.,3.6491470389141107);

\draw [line width=.8pt, dashed, color=black] (-7.,3.6491470389141107)--(-8.,1.9170962313452288)--(-9.,0.1850454237763482)-- (-10.,-1.5470053837925324);

\draw [line width=.8pt, dashed, color=black] (-3.,10.577350269189628)-- (-2.,8.845299461620748);

\draw [line width=.8pt, dashed, color=black] (-2.,8.845299461620748)--  (-1.,7.113248654051868);

\draw [line width=.8pt, dashed, color=black] (-1.,7.113248654051868)-- (0.,5.38119784648299);

\draw [line width=.8pt, dashed, color=black] (0.,5.38119784648299)-- (1.,3.6491470389141094);
\draw [line width=.8pt, dashed, color=black] (1.,3.6491470389141094)-- (2.,1.9170962313452288)--(3.,0.1850454237763482) --(4.,-1.5470053837925324);
%NN
\draw [line width=.8pt, dashed, color=black] (4.,-1.5470053837925324)-- (2.,-1.5470053837925324);
%NN
\draw [line width=.8pt, dashed, color=black] (-8.,-1.5470053837925324)-- (-10.,-1.5470053837925324);

\draw [line width=.8pt, color=black] (2.,-1.5470053837925324)-- (0.,-1.5470053837925324)-- (-2.,-1.5470053837925324)-- (-4.,-1.5470053837925324)-- (-6.,-1.5470053837925324)-- (-8.,-1.5470053837925324);

%NN
\draw [line width=.8pt,dashed,color=black] (1.,3.6491470389141094)-- (-1.,3.6491470389141134);
\draw [line width=.8pt,color=black] (-3.,3.6491470389141174)-- (-1.,3.6491470389141134);
\draw [line width=.8pt,color=black] (-3.,3.6491470389141174)--  (-5.,3.6491470389141125);
%NN
\draw [line width=.8pt,dashed,color=black] (-5.,3.6491470389141125)-- (-7.,3.6491470389141107);

%NN
\draw [line width=.8pt,dashed,color=black] (-4.,8.84529946162075)--(-2.,8.845299461620748)-- (-3.,7.113248654051869)--(-4.,8.84529946162075);
%NN
\draw [line width=.8pt, dashed,color=black] (-3.,7.113248654051869)-- (-5.,7.113248654051867);
%NN
\draw [line width=.8pt, dashed,color=black] 
 (-4.,5.381197846482992)-- (-5.,7.113248654051867);
\draw [line width=1.5pt,dotted,color=red]  (-4.,5.381197846482992)--(-3.,7.113248654051869);
\draw [line width=1.5pt,dotted,color=red] (-3.,7.113248654051869)--  (-2.,5.381197846482992);
%NN
\draw [line width=.8pt, dashed,color=black] 
 (-1.,7.113248654051869)--  (-2.,5.381197846482992);
%NN
\draw [line width=.8pt, dashed,color=black] 
(-3.,7.113248654051869)--  (-1.,7.113248654051868);
%NN
\draw [line width=.8pt, dashed,color=black] 
 (-6.,5.381197846482991)--(-4.,5.381197846482992);
%NN
\draw [line width=.8pt, dashed,color=black](-5.,3.6491470389141125)--(-6.,5.381197846482991);
\draw [line width=1.5pt,dotted,color=red](-5.,3.6491470389141125)--(-4.,5.381197846482992);
\draw [line width=.8pt, color=black] (-4.,5.381197846482992)--  (-2.,5.381197846482992);
\draw [line width=.8pt, color=black] (-3.,3.6491470389141174)--(-4.,5.381197846482992);
\draw [line width=.8pt, color=black] (-3.,3.6491470389141174)--(-2.,5.381197846482992);
%NN
\draw [line width=.8pt, dashed,color=black] (-2.,5.381197846482992)-- (0.,5.38119784648299)-- (-1.,3.6491470389141134);
\draw [line width=1.5pt,dotted,color=red] (-1.,3.6491470389141134)--(-2.,5.381197846482992);
\draw [line width=1.5pt,dotted,color=red] (-1.,3.6491470389141134)--(0.,1.9170962313452288)--(1.,0.1850454237763482);
\draw [line width=1.5pt,dotted,color=red] (-5.,3.6491470389141134)--(-6.,1.9170962313452288)--(-7.,0.1850454237763482);
%NN
\draw [line width=.8pt, dashed,color=black] 
 (1.,3.6491470389141134)--(0.,1.9170962313452288)--(2.,1.9170962313452288)--(1.,0.1850454237763482);
%NN
\draw [line width=.8pt, dashed,color=black] 
 (-7.,3.6491470389141134)--(-6.,1.9170962313452288)--(-8.,1.9170962313452288)--(-7.,0.1850454237763482);
\draw [line width=.8pt, color=black] (-1.,3.6491470389141134)--(-2.,1.9170962313452288)--(0.,1.9170962313452288)--(-1.,0.1850454237763482)--(-2.,1.9170962313452288)--(-3.,0.1850454237763482)--(-4.,1.9170962313452288)--(-5.,0.1850454237763482)--(-6.,1.9170962313452288)--(-4.,1.9170962313452288)--(-5.,3.6491470389141134);
\draw [line width=.8pt,color=black] (-3.,3.6491470389141174)--(-2.,1.9170962313452288)-- (-4.,1.9170962313452288)-- (-3.,3.6491470389141174);
\draw [line width=.8pt, color=black] (1.,0.1850454237763482)-- (-1.,0.1850454237763482)-- (-3.,0.1850454237763482)-- (-5.,0.1850454237763482) -- (-7.,0.1850454237763482);
%NN
\draw [line width=.8pt, dashed,color=black] 
(1.,0.1850454237763482)--(3.,0.1850454237763482)-- 
(2.,-1.5470053837925324);
\draw [line width=1.5pt, dotted,color=red](1.,0.1850454237763482)--(2.,-1.5470053837925324);
%NN
\draw [line width=.8pt, dashed,color=black] 
(1.,0.1850454237763482)--(3.,0.1850454237763482);
\draw [line width=.8pt, color=black](1.,0.1850454237763482)-- (0.,-1.5470053837925324)-- (-1.,0.1850454237763482)-- (-2.,-1.5470053837925324)-- (-3.,0.1850454237763482)-- (-4.,-1.5470053837925324)-- (-5.,0.1850454237763482) -- (-6.,-1.5470053837925324)-- (-7.,0.1850454237763482);
%NN
\draw [line width=.8pt, dashed,color=black](-8.,-1.5470053837925324)--(-9.,0.1850454237763482)-- (-7.,0.1850454237763482);
\draw [line width=1.5pt, dotted,color=red](-7.,0.1850454237763482)-- (-8.,-1.5470053837925324);%--
\begin{scriptsize}

\draw [fill=black] (-3.,10.577350269189628) circle (2.5pt);

\draw [fill=black] (-2.,5.381197846482992) circle (2.5pt);
 
\draw [fill=black] (-1.,7.113248654051868) circle (2.5pt);
 
\draw [fill=black] (-2.,8.845299461620748) circle (2.5pt);
 
\draw [fill=black] (-4.,8.84529946162075) circle (2.5pt);
 
\draw [fill=black] (-5.,7.113248654051867) circle (2.5pt);
 
\draw [fill=black] (-4.,5.381197846482992) circle (2.5pt);
 
\draw [fill=black] (-1.,3.6491470389141134) circle (2.5pt);
 
\draw [fill=black] (0.,5.38119784648299) circle (2.5pt);
 
\draw [fill=black] (1.,3.6491470389141094) circle (2.5pt);
 
\draw [fill=black] (-6.,5.381197846482991) circle (2.5pt);
 
\draw [fill=black] (-5.,3.6491470389141125) circle (2.5pt);
 
\draw [fill=black] (-7.,3.6491470389141107) circle (2.5pt);
 
\draw [fill=black] (-3.,3.6491470389141174) circle (2.5pt);

\draw [fill=black] (2.,1.9170962313452288) circle (2.5pt);

\draw [fill=black] (0.,1.9170962313452288) circle (2.5pt);

\draw [fill=black] (-2.,1.9170962313452288) circle (2.5pt);

\draw [fill=black] (-4.,1.9170962313452288) circle (2.5pt);

\draw [fill=black] (-6.,1.9170962313452288) circle (2.5pt);

\draw [fill=black] (-8.,1.9170962313452288) circle (2.5pt);

\draw [fill=black] (3.,0.1850454237763482) circle (2.5pt);

\draw [fill=black] (1.,0.1850454237763482) circle (2.5pt);

\draw [fill=black] (-1.,0.1850454237763482) circle (2.5pt);

\draw [fill=black] (-3.,0.1850454237763482) circle (2.5pt);

\draw [fill=black] (-5.,0.1850454237763482) circle (2.5pt);

\draw [fill=black] (-7.,0.1850454237763482) circle (2.5pt);

\draw [fill=black] (-9.,0.1850454237763482) circle (2.5pt);

\draw [fill=black] (4.,-1.5470053837925324) circle (2.5pt);

\draw [fill=black] (2.,-1.5470053837925324) circle (2.5pt);

\draw [fill=black] (0.,-1.5470053837925324) circle (2.5pt);

\draw [fill=black] (-2.,-1.5470053837925324) circle (2.5pt);

\draw [fill=black] (-4.,-1.5470053837925324) circle (2.5pt);

\draw [fill=black] (-6.,-1.5470053837925324) circle (2.5pt);

\draw [fill=black] (-8.,-1.5470053837925324) circle (2.5pt);

\draw [fill=black] (-10.,-1.5470053837925324) circle (2.5pt);

\end{scriptsize}

\end{tikzpicture}

\caption{Figure to assist with the proof of Lemma~\ref{lem:interior}.  In both panels the first seven rows of the $s$-subgrid is shown, and the boundary of the $s+1$ is shown in red.  The left panel shows the first seven rows of $T^{n,n-2,s}$ and the right panel shows the first seven rows of $T^{n,n-(s+1),s}$.  Dotted lines indicate those edges with values $c(s)$ (left) and $c(s+1)$ (right).  Solid lines have edge values equal to one.  Dashed lines have edge values not equal to one or $c(s+1)$ (they are not all equal). In the figure in the left panel a triangle on the boundary of the $s+1$-subgrid is highlighted in orange, with the parents of its left edge marked with filled circles.  A triangle in the interior is denoted by a blue triangle and the parent triangles of its left edge are marked with open circles.}
\label{fig:lemma53}
\end{center}
\end{figure}

Using an induction assumption, we now assume the lemma holds for $s$ and proceed to show that it holds for the case $s+1$.  We consider two cases of edges, according to whether the edges is in $Int(T^c)$ or $\partial^2(T^c)$. See Figure~\ref{fig:lemma53} for illustrative clarification. 
	
\textbf{Case of  Int($T_c$):}  If edge $e$ is in  $Int(T^c)$ then by (iii), the edge values in $T^p$ determining it belong to $Int(T^p).$ By (i) $Int(T^p)=1.$ Hence, by the functions associated with the lemmas of Section~\ref{sec:proofmethods}, the resistance value is one.

\textbf{Case of $\partial^2 T^c$:} Recall that by symmetry,  the edge values on the left side of $\partial^2 T^c$ are identical to the edge values on the right and base side. So it suffices to assume we are dealing with a left-edge $e \in \partial^2 T^c.$ The label of this edge is determined by Lemma~\ref{lem:3trianglesleft}. There are 9 arguments determining $e$, 8 of which are 1, and one of which, by (ii) is a constant. Therefore the value of $e$ is also constant and also dependent on $s+1.$ This completes the proof.
\end{proof} 
 
Lemma~\ref{lem:interior}(i) gives rise to the heuristic that the ``center" of the grid $T^{n,n-i,i+1}$ is uniformly 1. Several other uniformity patterns are studied in the next section resulting in the statement of the  Uniform Center Theorem.

 \section{The Uniform Center Theorem}\label{sec:constcenter}
This section generalizes results naturally motivated by the two results in Section 5. Prior to doing this we clarify the concept of type and pre-type defined in Definition~\ref{def:pretype} with the following example.
 \begin{example}    By Corollary \ref{cor:tnnminus1istwothirds}, the  pre-type of $T^{n,n-1}_{1,1}$ is $(\frac{2}{3},\frac{2}{3},1).$ Similarly,  the pre-type of
 $T^{n,n-1}_{2,1}$ is 
 $(\frac{2}{3},1, 1),$ (provided $n \ge 3$) and  the pre-type of $T^{n,n-1}_{2,2}$ is $(1, \frac{2}{3},1).$  The ordered triplets associated with $T^{n,n-1}_{2,d}, d \in \{1,2\}$ are distinct; hence, the triangles are not of the same pre-type but are of the same type. \end{example}

 When proving results, we will prove them for the pre-types in the upper left half; by Lemma \ref{lem:upperlefthalf} the results then extend to all triangles in the underlying $n$-grid. This heuristic will be amply illustrated in the remainder of this section.

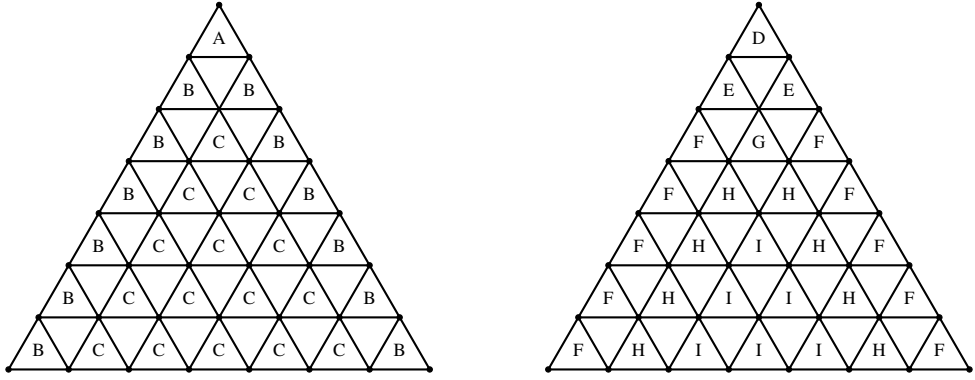
\begin{figure}[ht!]
\begin{center}
\begin{tikzpicture}[scale = .5, line cap=round,line join=round,>=triangle 45,x=1.0cm,y=1.0cm,scale=.8]

\draw [line width=.8pt, color=black] (-3.,10.577350269189628)--  (-4.,8.84529946162075);

\draw [line width=.8pt, color=black] (-4.,8.84529946162075)--  (-5.,7.113248654051867);

\draw [line width=.8pt, color=black] (-5.,7.113248654051867)--  (-6.,5.381197846482991);

\draw [line width=.8pt, color=black] (-6.,5.381197846482991)-- (-7.,3.6491470389141107);

\draw [line width=.8pt, color=black] (-7.,3.6491470389141107)--(-8.,1.9170962313452288)--(-9.,0.1850454237763482)-- (-10.,-1.5470053837925324);

\draw [line width=.8pt, color=black] (-3.,10.577350269189628)-- (-2.,8.845299461620748);

\draw [line width=.8pt, color=black] (-2.,8.845299461620748)--  (-1.,7.113248654051868);

\draw [line width=.8pt, color=black] (-1.,7.113248654051868)-- (0.,5.38119784648299);

\draw [line width=.8pt, color=black] (0.,5.38119784648299)-- (1.,3.6491470389141094);
\draw [line width=.8pt, color=black] (1.,3.6491470389141094)-- (2.,1.9170962313452288)--(3.,0.1850454237763482) --(4.,-1.5470053837925324);
\draw [line width=.8pt, color=black] (4.,-1.5470053837925324)-- (2.,-1.5470053837925324)-- (0.,-1.5470053837925324)-- (-2.,-1.5470053837925324)-- (-4.,-1.5470053837925324)-- (-6.,-1.5470053837925324)-- (-8.,-1.5470053837925324)-- (-10.,-1.5470053837925324);
\draw [line width=.8pt,color=black] (1.,3.6491470389141094)-- (-1.,3.6491470389141134);
\draw [line width=.8pt, color=black] (-3.,3.6491470389141174)-- (-1.,3.6491470389141134);
\draw [line width=.8pt, color=black] (-3.,3.6491470389141174)--  (-5.,3.6491470389141125);
\draw [line width=.8pt, color=black] (-5.,3.6491470389141125)-- (-7.,3.6491470389141107);

\draw [line width=.8pt,color=black] (-4.,8.84529946162075)--(-2.,8.845299461620748)-- (-3.,7.113248654051869)--(-4.,8.84529946162075);
\draw [line width=.8pt,color=black] (-3.,7.113248654051869)-- (-5.,7.113248654051867);
\draw [line width=.8pt, color=black] (-4.,5.381197846482992)-- (-5.,7.113248654051867);
\draw [line width=.8pt,color=black]  (-4.,5.381197846482992)--(-3.,7.113248654051869);
\draw [line width=.8pt,color=black] (-3.,7.113248654051869)--  (-2.,5.381197846482992);
\draw [line width=.8pt, color=black] (-1.,7.113248654051869)--  (-2.,5.381197846482992);
\draw [line width=.8pt,color=black] (-1.,7.113248654051868)--(-3.,7.113248654051869);
\draw [line width=.8pt, color=black] (-6.,5.381197846482991)--(-4.,5.381197846482992);
\draw [line width=.8pt, color=black](-5.,3.6491470389141125)--(-6.,5.381197846482991);
\draw [line width=.8pt,color=black](-5.,3.6491470389141125)--(-4.,5.381197846482992);
\draw [line width=.8pt, color=black] (-4.,5.381197846482992)--  (-2.,5.381197846482992);
\draw [line width=.8pt,color=black] (-3.,3.6491470389141174)--(-4.,5.381197846482992);
\draw [line width=.8pt, color=black] (-3.,3.6491470389141174)--(-2.,5.381197846482992);
\draw [line width=.8pt, color=black] (-2.,5.381197846482992)-- (0.,5.38119784648299)-- (-1.,3.6491470389141134);
\draw [line width=.8pt,color=black] (-1.,3.6491470389141134)--(-2.,5.381197846482992);
\draw [line width=.8pt,color=black] (-1.,3.6491470389141134)--(0.,1.9170962313452288)--(1.,0.1850454237763482);
\draw [line width=.8pt,color=black] (-5.,3.6491470389141134)--(-6.,1.9170962313452288)--(-7.,0.1850454237763482);
\draw [line width=.8pt, color=black] (1.,3.6491470389141134)--(0.,1.9170962313452288)--(2.,1.9170962313452288)--(1.,0.1850454237763482);
\draw [line width=.8pt, color=black] (-7.,3.6491470389141134)--(-6.,1.9170962313452288)--(-8.,1.9170962313452288)--(-7.,0.1850454237763482);
\draw [line width=.8pt, color=black] (-1.,3.6491470389141134)--(-2.,1.9170962313452288)--(0.,1.9170962313452288)--(-1.,0.1850454237763482)--(-2.,1.9170962313452288)--(-3.,0.1850454237763482)--(-4.,1.9170962313452288)--(-5.,0.1850454237763482)--(-6.,1.9170962313452288)--(-4.,1.9170962313452288)--(-5.,3.6491470389141134);
\draw [line width=.8pt,color=black] (-3.,3.6491470389141174)--(-2.,1.9170962313452288)-- (-4.,1.9170962313452288)-- (-3.,3.6491470389141174);
\draw [line width=.8pt, color=black] (1.,0.1850454237763482)-- (-1.,0.1850454237763482)-- (-3.,0.1850454237763482)-- (-5.,0.1850454237763482) -- (-7.,0.1850454237763482);
\draw [line width=.8pt, color=black](1.,0.1850454237763482)--(3.,0.1850454237763482)-- (2.,-1.5470053837925324)--(1.,0.1850454237763482)-- (0.,-1.5470053837925324)-- (-1.,0.1850454237763482)-- (-2.,-1.5470053837925324)-- (-3.,0.1850454237763482)-- (-4.,-1.5470053837925324)-- (-5.,0.1850454237763482) -- (-6.,-1.5470053837925324)-- (-7.,0.1850454237763482)-- (-8.,-1.5470053837925324)--(-9.,0.1850454237763482) -- (-7.,0.1850454237763482);
%--
\begin{scriptsize}

\draw [fill=black] (-3.,10.577350269189628) circle (2.5pt);

\draw [fill=black] (-2.,5.381197846482992) circle (2.5pt);
 
\draw [fill=black] (-1.,7.113248654051868) circle (2.5pt);
 
\draw [fill=black] (-2.,8.845299461620748) circle (2.5pt);
 
\draw [fill=black] (-4.,8.84529946162075) circle (2.5pt);
 
\draw [fill=black] (-5.,7.113248654051867) circle (2.5pt);
 
\draw [fill=black] (-4.,5.381197846482992) circle (2.5pt);
 
\draw [fill=black] (-1.,3.6491470389141134) circle (2.5pt);
 
\draw [fill=black] (0.,5.38119784648299) circle (2.5pt);
 
\draw [fill=black] (1.,3.6491470389141094) circle (2.5pt);
 
\draw [fill=black] (-6.,5.381197846482991) circle (2.5pt);
 
\draw [fill=black] (-5.,3.6491470389141125) circle (2.5pt);
 
\draw [fill=black] (-7.,3.6491470389141107) circle (2.5pt);
 
\draw [fill=black] (-3.,3.6491470389141174) circle (2.5pt);

\draw [fill=black] (2.,1.9170962313452288) circle (2.5pt);

\draw [fill=black] (0.,1.9170962313452288) circle (2.5pt);

\draw [fill=black] (-2.,1.9170962313452288) circle (2.5pt);

\draw [fill=black] (-4.,1.9170962313452288) circle (2.5pt);

\draw [fill=black] (-6.,1.9170962313452288) circle (2.5pt);

\draw [fill=black] (-8.,1.9170962313452288) circle (2.5pt);

\draw [fill=black] (3.,0.1850454237763482) circle (2.5pt);

\draw [fill=black] (1.,0.1850454237763482) circle (2.5pt);

\draw [fill=black] (-1.,0.1850454237763482) circle (2.5pt);

\draw [fill=black] (-3.,0.1850454237763482) circle (2.5pt);

\draw [fill=black] (-5.,0.1850454237763482) circle (2.5pt);

\draw [fill=black] (-7.,0.1850454237763482) circle (2.5pt);

\draw [fill=black] (-9.,0.1850454237763482) circle (2.5pt);

\draw [fill=black] (4.,-1.5470053837925324) circle (2.5pt);

\draw [fill=black] (2.,-1.5470053837925324) circle (2.5pt);

\draw [fill=black] (0.,-1.5470053837925324) circle (2.5pt);

\draw [fill=black] (-2.,-1.5470053837925324) circle (2.5pt);

\draw [fill=black] (-4.,-1.5470053837925324) circle (2.5pt);

\draw [fill=black] (-6.,-1.5470053837925324) circle (2.5pt);

\draw [fill=black] (-8.,-1.5470053837925324) circle (2.5pt);

\draw [fill=black] (-10.,-1.5470053837925324) circle (2.5pt);
\node at (-3,9.5) (nodeA) {A};
\node at (-4,7.77) (nodeB) {B};
\node at (-2,7.77) (nodeB2) {B};
\node at (-1,6.04) (nodeC1) {B};
\node at (-3,6.04) (nodeD1) {C};
\node at (-5,6.04) (nodeC2) {B};
\node at (0,4.30) (nodeE1) {B};
\node at (-2,4.30) (nodeF1) {C};
\node at (-4,4.30) (nodeF2) {C};
\node at (-6,4.30) (nodeE2) {B};
\node at (1,2.57) {B};
\node at (-1,2.57) {C};
\node at (-3,2.57) {C};
\node at (-5,2.57) {C};
\node at (-7,2.57) {B};
\node at (2,0.84) {B};
\node at (0,0.84) {C};
\node at (-2,0.84) {C};
\node at (-4,0.84) {C};
\node at (-6,0.84) {C};
\node at (-8,0.84) {B};
\node at (3,-0.89) {B};
\node at (1,-0.89) {C};
\node at (-1,-0.89) {C};
\node at (-3,-0.89) {C};
\node at (-5,-0.89) {C};
\node at (-7,-0.89) {C};
\node at (-9,-0.89) {B};
\end{scriptsize}

\end{tikzpicture}
\qquad\qquad \begin{tikzpicture}[scale = .5, line cap=round,line join=round,>=triangle 45,x=1.0cm,y=1.0cm,scale=.8]

\draw [line width=.8pt, color=black] (-3.,10.577350269189628)--  (-4.,8.84529946162075);

\draw [line width=.8pt, color=black] (-4.,8.84529946162075)--  (-5.,7.113248654051867);

\draw [line width=.8pt, color=black] (-5.,7.113248654051867)--  (-6.,5.381197846482991);

\draw [line width=.8pt, color=black] (-6.,5.381197846482991)-- (-7.,3.6491470389141107);

\draw [line width=.8pt, color=black] (-7.,3.6491470389141107)--(-8.,1.9170962313452288)--(-9.,0.1850454237763482)-- (-10.,-1.5470053837925324);

\draw [line width=.8pt, color=black] (-3.,10.577350269189628)-- (-2.,8.845299461620748);

\draw [line width=.8pt, color=black] (-2.,8.845299461620748)--  (-1.,7.113248654051868);

\draw [line width=.8pt, color=black] (-1.,7.113248654051868)-- (0.,5.38119784648299);

\draw [line width=.8pt, color=black] (0.,5.38119784648299)-- (1.,3.6491470389141094);
\draw [line width=.8pt, color=black] (1.,3.6491470389141094)-- (2.,1.9170962313452288)--(3.,0.1850454237763482) --(4.,-1.5470053837925324);
\draw [line width=.8pt, color=black] (4.,-1.5470053837925324)-- (2.,-1.5470053837925324)-- (0.,-1.5470053837925324)-- (-2.,-1.5470053837925324)-- (-4.,-1.5470053837925324)-- (-6.,-1.5470053837925324)-- (-8.,-1.5470053837925324)-- (-10.,-1.5470053837925324);
\draw [line width=.8pt,color=black] (1.,3.6491470389141094)-- (-1.,3.6491470389141134);
\draw [line width=.8pt, color=black] (-3.,3.6491470389141174)-- (-1.,3.6491470389141134);
\draw [line width=.8pt, color=black] (-3.,3.6491470389141174)--  (-5.,3.6491470389141125);
\draw [line width=.8pt, color=black] (-5.,3.6491470389141125)-- (-7.,3.6491470389141107);

\draw [line width=.8pt,color=black] (-4.,8.84529946162075)--(-2.,8.845299461620748)-- (-3.,7.113248654051869)--(-4.,8.84529946162075);
\draw [line width=.8pt,color=black] (-3.,7.113248654051869)-- (-5.,7.113248654051867);
\draw [line width=.8pt, color=black] (-4.,5.381197846482992)-- (-5.,7.113248654051867);
\draw [line width=.8pt,color=black]  (-4.,5.381197846482992)--(-3.,7.113248654051869);
\draw [line width=.8pt,color=black] (-3.,7.113248654051869)--  (-2.,5.381197846482992);
\draw [line width=.8pt, color=black] (-1.,7.113248654051869)--  (-2.,5.381197846482992);
\draw [line width=.8pt,color=black] (-1.,7.113248654051868)--(-3.,7.113248654051869);
\draw [line width=.8pt, color=black] (-6.,5.381197846482991)--(-4.,5.381197846482992);
\draw [line width=.8pt, color=black](-5.,3.6491470389141125)--(-6.,5.381197846482991);
\draw [line width=.8pt,color=black](-5.,3.6491470389141125)--(-4.,5.381197846482992);
\draw [line width=.8pt, color=black] (-4.,5.381197846482992)--  (-2.,5.381197846482992);
\draw [line width=.8pt,color=black] (-3.,3.6491470389141174)--(-4.,5.381197846482992);
\draw [line width=.8pt, color=black] (-3.,3.6491470389141174)--(-2.,5.381197846482992);
\draw [line width=.8pt, color=black] (-2.,5.381197846482992)-- (0.,5.38119784648299)-- (-1.,3.6491470389141134);
\draw [line width=.8pt,color=black] (-1.,3.6491470389141134)--(-2.,5.381197846482992);
\draw [line width=.8pt,color=black] (-1.,3.6491470389141134)--(0.,1.9170962313452288)--(1.,0.1850454237763482);
\draw [line width=.8pt,color=black] (-5.,3.6491470389141134)--(-6.,1.9170962313452288)--(-7.,0.1850454237763482);
\draw [line width=.8pt, color=black] (1.,3.6491470389141134)--(0.,1.9170962313452288)--(2.,1.9170962313452288)--(1.,0.1850454237763482);
\draw [line width=.8pt, color=black] (-7.,3.6491470389141134)--(-6.,1.9170962313452288)--(-8.,1.9170962313452288)--(-7.,0.1850454237763482);
\draw [line width=.8pt, color=black] (-1.,3.6491470389141134)--(-2.,1.9170962313452288)--(0.,1.9170962313452288)--(-1.,0.1850454237763482)--(-2.,1.9170962313452288)--(-3.,0.1850454237763482)--(-4.,1.9170962313452288)--(-5.,0.1850454237763482)--(-6.,1.9170962313452288)--(-4.,1.9170962313452288)--(-5.,3.6491470389141134);
\draw [line width=.8pt,color=black] (-3.,3.6491470389141174)--(-2.,1.9170962313452288)-- (-4.,1.9170962313452288)-- (-3.,3.6491470389141174);
\draw [line width=.8pt, color=black] (1.,0.1850454237763482)-- (-1.,0.1850454237763482)-- (-3.,0.1850454237763482)-- (-5.,0.1850454237763482) -- (-7.,0.1850454237763482);
\draw [line width=.8pt, color=black](1.,0.1850454237763482)--(3.,0.1850454237763482)-- (2.,-1.5470053837925324)--(1.,0.1850454237763482)-- (0.,-1.5470053837925324)-- (-1.,0.1850454237763482)-- (-2.,-1.5470053837925324)-- (-3.,0.1850454237763482)-- (-4.,-1.5470053837925324)-- (-5.,0.1850454237763482) -- (-6.,-1.5470053837925324)-- (-7.,0.1850454237763482)-- (-8.,-1.5470053837925324)--(-9.,0.1850454237763482) -- (-7.,0.1850454237763482);

%--
\begin{scriptsize}

\draw [fill=black] (-3.,10.577350269189628) circle (2.5pt);

\draw [fill=black] (-2.,5.381197846482992) circle (2.5pt);
 
\draw [fill=black] (-1.,7.113248654051868) circle (2.5pt);
 
\draw [fill=black] (-2.,8.845299461620748) circle (2.5pt);
 
\draw [fill=black] (-4.,8.84529946162075) circle (2.5pt);
 
\draw [fill=black] (-5.,7.113248654051867) circle (2.5pt);
 
\draw [fill=black] (-4.,5.381197846482992) circle (2.5pt);
 
\draw [fill=black] (-1.,3.6491470389141134) circle (2.5pt);
 
\draw [fill=black] (0.,5.38119784648299) circle (2.5pt);
 
\draw [fill=black] (1.,3.6491470389141094) circle (2.5pt);
 
\draw [fill=black] (-6.,5.381197846482991) circle (2.5pt);
 
\draw [fill=black] (-5.,3.6491470389141125) circle (2.5pt);
 
\draw [fill=black] (-7.,3.6491470389141107) circle (2.5pt);
 
\draw [fill=black] (-3.,3.6491470389141174) circle (2.5pt);

\draw [fill=black] (2.,1.9170962313452288) circle (2.5pt);

\draw [fill=black] (0.,1.9170962313452288) circle (2.5pt);

\draw [fill=black] (-2.,1.9170962313452288) circle (2.5pt);

\draw [fill=black] (-4.,1.9170962313452288) circle (2.5pt);

\draw [fill=black] (-6.,1.9170962313452288) circle (2.5pt);

\draw [fill=black] (-8.,1.9170962313452288) circle (2.5pt);

\draw [fill=black] (3.,0.1850454237763482) circle (2.5pt);

\draw [fill=black] (1.,0.1850454237763482) circle (2.5pt);

\draw [fill=black] (-1.,0.1850454237763482) circle (2.5pt);

\draw [fill=black] (-3.,0.1850454237763482) circle (2.5pt);

\draw [fill=black] (-5.,0.1850454237763482) circle (2.5pt);

\draw [fill=black] (-7.,0.1850454237763482) circle (2.5pt);

\draw [fill=black] (-9.,0.1850454237763482) circle (2.5pt);

\draw [fill=black] (4.,-1.5470053837925324) circle (2.5pt);

\draw [fill=black] (2.,-1.5470053837925324) circle (2.5pt);

\draw [fill=black] (0.,-1.5470053837925324) circle (2.5pt);

\draw [fill=black] (-2.,-1.5470053837925324) circle (2.5pt);

\draw [fill=black] (-4.,-1.5470053837925324) circle (2.5pt);

\draw [fill=black] (-6.,-1.5470053837925324) circle (2.5pt);

\draw [fill=black] (-8.,-1.5470053837925324) circle (2.5pt);

\draw [fill=black] (-10.,-1.5470053837925324) circle (2.5pt);
\node at (-3,9.5) (nodeA) {D};
\node at (-4,7.77) (nodeB) {E};
\node at (-2,7.77) (nodeB2) {E};
\node at (-1,6.04) (nodeC1) {F};
\node at (-3,6.04) (nodeD1) {G};
\node at (-5,6.04) (nodeC2) {F};
\node at (0,4.30) (nodeE1) {F};
\node at (-2,4.30) (nodeF1) {H};
\node at (-4,4.30) (nodeF2) {H};
\node at (-6,4.30) (nodeE2) {F};
\node at (1,2.57) {F};
\node at (-1,2.57) {H};
\node at (-3,2.57) {I};
\node at (-5,2.57) {H};
\node at (-7,2.57) {F};
\node at (2,0.84) {F};
\node at (0,0.84) {H};
\node at (-2,0.84) {I};
\node at (-4,0.84) {I};
\node at (-6,0.84) {H};
\node at (-8,0.84) {F};
\node at (3,-0.89) {F};
\node at (1,-0.89) {H};
\node at (-1,-0.89) {I};
\node at (-3,-0.89) {I};
\node at (-5,-0.89) {I};
\node at (-7,-0.89) {H};
\node at (-9,-0.89) {F};
\end{scriptsize}

\end{tikzpicture}

\caption{The triangle types of the first 7 rows of $T^{n}$ after 1 reduction (left) and 2 reductions (right) of $T^n.$
Distinct letters correspond to distinct types.
The figure is true if   $n$ is  large, say $n \ge 23.$ Consequently, the base of the grid is not shown.}
\label{fig:7gridtriangletypes}
\end{center}
\end{figure}

To motivate the Uniform Center Theorem we first observe certain patterns in triangle types as illustrated in Figure~\ref{fig:7gridtriangletypes} which shows the first seven rows of $T^{n,n-1}$ and $T^{n,n-2}$ for any sufficiently large $n$. The left hand panel, shows that in $T^{n,n-1}$  that for rows, $2 \le r \le 7$ on diagonal $d=1$ and for $3 \le r \ge 7,$ on $d=2$ the types are  identical (triangles B and C in the figure). Similarly, the right hand panel presenting $T^{n,n-2}$ shows that for $3 \le r \le  7, d=1; 4 \le r \le 7, d=2; \text{ and } 5 \le r \le 7, d=3$ the types are identical (F, G, and I in the figure).   These patterns generalize as follows:  For $n \ge 4s, 1 \le d \le s,$ the triangles $T^{n,n-s}_{r,d}$  where $s+d \le r \le n-2s$ all have the same pre-type. By symmetry, the triangles $T^{n,n-s}_{r,r+1-s}$ also all have the same pre-type.  These observations and patterns motivate the concept of the uniform center.
 
%\color{red} This seems to be the only place  in the narrative where angle bracket notation is used. So lets get rid of it (We can replace it by the T notation). I think we agreed that the angle bracket notation can be left in the figures but we should make this notation explicit in the first figure it occurs  (Sections 4 and 5) I think we can have this convention in the captions. \color{black}  
\begin{definition}\label{def:constantcenter} Fix  $s \ge 1.$ If $n \ge 4s, 1 \le d \le s,$  the uniform center of $T^{n,n-s}$ refers to the union of the following three sets of triangles:
  \begin{equation}\label{equ:constantcenterrange}
  T^{n,n-s}_ {r, d}, \qquad s+d \le r \le n-2s,
 \end{equation}
 \begin{equation}\label{equ:constantcenterrange2} 
 T^{n,n-s}_ {r, r+1-d}, \qquad s+d \le r \le n-2s,
 \end{equation} \begin{equation}\label{equ:constantcenterrange3}
 T^{n,n-s}_ {n-s-d+1, p}, \quad s+d \leq p \leq n-2s.\end{equation}
 \end{definition}
 
 The theorem statements will be complete and refer to all sides of the underlying $n$-grid. However, in the proofs we will only deal with \eqref{equ:constantcenterrange} which lies in the upper left half. By our remarks earlier about symmetry, the proofs extend to the entire uniform center. 

The Uniform Center Theorem simply captures the heuristic that the types of the triangles on each diagonal $d, 1 \le d \le s$ are equal in a central region of $T^{n,n-s}.$  Unexpectedly, the theorem can be proven without explicit computation using the methods of Section \ref{sec:proofmethods}. Prior  to stating the theorem, we present two easily proven lemmas which will be useful in the inductive proof of the theorem.

\begin{lemma}\label{lem:inductionassumption} For $s=1, n \ge 3, \; s+d \le r \le n-2s,$ and $
 1 \le d \le s+1$ the pre-types of $T^{n,n-s}_{r,d}$ are identical. 

\end{lemma}
\begin{proof} By Lemma 2.12 it suffices to prove the result for the upper left half.  If $d=1,$ then by Corollary 5.1,  the  pre-type of triangle $T^{n,n-1}_{r,1}, 2 \le r \le n-2,$
is identically $\{\frac{2}{3},1,1\}$. Also by Corollary 5.1, if $d\geq 2$ the pre-type of triangle $T^{n,n-1}_{r,2}$
$3 \le r \le n-2,$ is identically $\{1,1,1\}.$  
\end{proof}

\begin{lemma}\label{lem:inductionstep}
For $n \ge 4s, s \ge 1,$ 
$$
\text{The pre-types of } T^{n,n-s}_{r,d}, \qquad s+d \le r \le n-2s, 
\qquad  s \le d \le s+1,  \text{ are identical }. 
$$
\end{lemma}
\begin{proof}
We first recall some facts about the $s$-grid. By Definition 2.13, the top and lower left corners of the subgrid $T^{n,n-s,s}$ are located at $r=2s-1, d=s$ and $r=n-2s+1,d=s$ respectively. By Lemma~\ref{lem:interior} and considerations of symmetry, in the upper left half, (i) the pre-types of the corner triangles are $\{c(s), c(s),1 \},$ (ii) the pre-types of non-corner triangles on the left side of the triangle boundary are $\{c(s),1,1\},$ and (iii) the pre-types of triangles not on the $s$-grid boundary are $\{1,1,1\}$ where $c(s)$ is as in Lemma~\ref{lem:interior}. 

It immediately follows that the pre-types (on the upper left half) of $T^{n,n-s}_{r,d}, d=s, 2s \le r \le n-2s$ are identically $\{c(s),1,1\}.$ It also immediately follows that the pre-types (on the upper left half) of 
$T^{n,n-s}_{r,d}, d=s+1, 2s+1 \le r \le n-2s$ are identically $\{1,1,1\}.$
\end{proof}
 
  \begin{theorem}[Uniform Centers]\label{the:constantcenters} For any  $s \ge 1,$ and for   $n \ge 4s,$ 
  \begin{enumerate}
\item[(a)] In the uniform center each of the following three sets of triangles $\langle r,d \rangle_{n,n-s}, \langle r,(d+1)-d \rangle_{n,n-s}$, and $\langle n-s-d+1, p \rangle$ have a single pre-type.  
\item[(b)] The triangles   satisfying \eqref{equ:constantcenterrange}, \eqref{equ:constantcenterrange2}, or \eqref{equ:constantcenterrange3}, are of the same type. 
\item[(c)] For each defined $s$-subgrid $T^{n,n-s,s}$ ( Definition \ref{def:ssubgrid}) the interior edge labels are identically 1, and the boundary edge resistance labels are all equal.   
 \item[(d)] For triangles in the uniform center satisfying \eqref{equ:constantcenterrange} the right and base resistance labels are identical, for triangles satisfying~\eqref{equ:constantcenterrange2} the left and base resistance labels are identical, and for triangles satisfying~\eqref{equ:constantcenterrange3} the left and right resistance labels are identical. 
 \end{enumerate}
  \end{theorem}

\begin{proof} 
 Clearly part (b) follows from part (a). Parts (c) and (d) are simply a restatement of Lemma~\ref{lem:interior} and is included in the theorem for purposes of completeness. Hence, it suffices to  prove part (a) and by remarks made above it suffices to prove assertions for the upper left half.
The proof uses an induction argument on $s.$
The base case, $s=1$ is provided by Lemma \ref{lem:inductionassumption}. 

For an induction assumption we assume that for some $s \ge 1,$ that
\begin{equation}\label{equ:uc-toprove}
\text{ for given $d,s,$ the pre-types of $T^{n,n-s}_{r,d}$ are identical for $s+d \le r \le n-2s, 1 \le d \le s+1.$}
\end{equation}
We proceed to prove this assertion with $s$ replaced by $s+1.$ There are two cases to consider according to whether $d \le s$ or $d \in \{s+1,s+2\}.$

\textbf{Case $d \in \{s+1,s+2\}$.} This follows from Lemma \ref{lem:inductionstep} with $s$ replaced by $s+1.$

\textbf{Case $1 \le d \le s.$} To prove the pre-types in \eqref{equ:uc-toprove} are identical, we must show that the right, boundary, left-boundary (if applicable), and non-boundary left (if applicable) edges are identically labeled. We do this using the lemmas presented in Section 4. We suffice with showing that all base edges are identically labeled on each diagonal, the proofs for the other edges being similar and hence omitted.

By the Base Edge Lemma and the notation introduced in Section 4, for given $d, 1 \le d \le s$ and $r, s+1+d \le r \le n-2(s+1),$ we have
 \begin{equation}\label{equ:uc-toprovebase}
 T^{n,n-(s+1)}_{r,d,B} =F(T^{n,n-s}_{r+1,d},T^{n,n-s}_{r+1,d+1},T^{n,n-s}_{r+2,d+1}).
 \end{equation}
By \eqref{equ:uc-toprove}, the triangle arguments on the right-hand side of this last equation have edges whose values are independent of $r.$ Hence the value of the left hand side is also independent of $r$ as was to be shown.

To illustrate the subtleties of how the induction assumption justifies the transition from $s$ to $s+1,$ note that the upper bound for $r$ in \eqref{equ:uc-toprove} is $n-2s,$ and that the maximum row value in the arguments on the right hand side of \eqref{equ:uc-toprovebase}  is $r+2.$ But $n-2s+2=n-2(s+1),$ illustrating how the induction assumption justifies the transition from $s$ to $s+1.$

%    The proof of the Uniform Center theorem is complete.
\end{proof}
 
 \begin{corollary}\label{cor:provenlater} Conjecture \ref{con:main}(b) is true. \end{corollary}
 
 As indicated earlier, Conjecture \ref{con:main} motivated the other results in this paper.  Corollary \ref{cor:provenlater} shows how the Uniform Centers Theorem which is used to prove this corollary could have been partially motivated by Conjecture \ref{con:main}. The next section provides a further example of how the Uniform Centers Theorem partially motivates the Main Conjecture. 
 
\section{Proof of the Vanishing One Conjecture in a Special Case} \label{sec:edgeval}

While we have proven the first two parts of Conjecture~\ref{con:main}, a proof of the remaining two parts remains elusive.  It is, however, possible to show the proof of parts (c) and (d) for certain edges, specifically those boundary edges in the first subgrid of  reduced triangles.  In this section, we use the notation
\[L_s = T_{s, 1, L}^{n,n-s},
\qquad \text{ for } n \ge 4s,\]
to refer to the left boundary edge of the triangle at row $s$ diagonal 1, of $T^{n, n-s}$ which by the Uniform Center Theorem is the unique edge label of the uniform center of diagonal 1 in the $s$-th reduction of the original $n$ grid. In a similar manner we define
\[B_s= T_{s, 1, B}^{n,n-s}, \qquad \text{ for } n \ge 4s. \]

\begin{lemma}\label{lem:2leftbdy} The sequence $L_s, s=1,2,3, \dotsc$  is monotone decreasing. \end{lemma}
\begin{proof} Using the Boundary Edge Lemma, and its accompanying Figure \ref{fig:2triangles}, as summarized in Figure \ref{fig:monotone} we have 
\[
    L_{s+1} = F(A,A)=\frac{2L_{s} B_{s}}
                    {L_s+2 B_s}.
\]
\noindent However, clearing denominators and gathering like terms, this last equation is equivalent to the assertion that
\begin{equation}\label{equ:LsBs} L_{s+1} L_{s} = 2 B_{s} (L_{s}-L_{s+1}).
\end{equation}
The proof is completed by noting that i) all edge label resistances are positive and consequently, ii) the left side of the last equation is positive implying that iii) the parenthetical expression on the right side is positive which iv) is equivalent to a statement of monotonicity.
\end{proof}
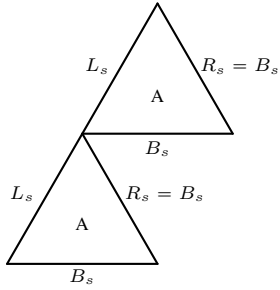
\begin{figure}[ht!]
\begin{center}
\begin{tikzpicture}[scale = 1, line cap=round,line join=round,>=triangle 45,x=1.0cm,y=1.0cm,scale = 1]
\draw [line width=.8pt,color=black] (-2.,10.)-- (-3.,11.732050807568879);
\draw [line width=.8pt,color=black] (-2.,10.)-- (-4.,10);
\draw [line width=.8pt,color=black] (-3.,11.732050807568879)-- (-4.,10.);
\draw [line width=.8pt,color=black] (-4.,10.)-- (-3.,8.267949192431121);
%\draw [line width=.8pt,color=black] (-3.,8.267949192431121)-- (-2.,10.);
\draw [line width=.8pt,color=black] (-4.,10.)-- (-5.,8.267949192431123);
\draw [line width=.8pt,color=black] (-5.,8.267949192431123)-- (-3.,8.267949192431121);
%\draw [line width=.8pt,color=black] (-3.,8.267949192431121)-- (-1.,8.267949192431121);
%\draw [line width=.8pt,color=black] (-1.,8.267949192431121)-- (-2.,10.);

\begin{scriptsize}
\draw[color=black] (-3, 10.5) node {A};
\draw[color=black] (-4,8.8) node {A};
\draw[color=black] (-3.8, 10.9) node {$L_s$};
\draw[color=black] (-4.8,9.2) node {$L_s$};
\draw[color=black] (-1.9, 10.9) node {$R_s=B_s$};
\draw[color=black] (-2.9,9.2) node {$R_s=B_s$};
\draw[color=black] (-3, 9.8) node {$B_s$};
\draw[color=black] (-4,8.1) node {$B_s$};

\end{scriptsize}
\end{tikzpicture}
\end{center}
\caption{Two left boundary triangles referred to in Lemma ~\ref{lem:2leftbdy}.  These two triangles appear in $T^{n,n-s}$ and are in row $r$ and $r+1$ where $s+1 \le r \le n-2s$ (i.e., these two triangles are in the uniform center of diagonal 1 after $s$ reductions).}
\label{fig:monotone}
\end{figure}

The following corollary proves part (c)  of Conjecture \ref{con:main} for these boundary edges for sufficiently large $n.$
\begin{corollary}\label{cor:specialcase} For all $s,$  $L_{s} <1.$ \end{corollary}
\begin{proof} The corollary follows from the lemma and the fact that
$L_{1} = \frac{2}{3}$ (Corollary \ref{cor:tnnminus1istwothirds}).   \end{proof}

We note that \eqref{equ:LsBs} implies that    the  sequence $(L_s)$ in the constant center determines the edge labels of the other edges   in the uniform center for the diagonal $d=1.$

\section{Conclusion}\label{sec:open}

The following section presents one further conjecture as well as directions for future investigations.

\begin{conjecture}\label{con:gcd}
%\color{red} Referee objected to notation $L_s^i.$ We only odefined $L_s$ which refers to the first diagonal. Perhaps easiest is 
%$L_s^1 = L_s$ and $L_s^2$ defineed as the left edge on diagonal 2 (needs a definition) \color{black}
Define $L_s^{(1)} = T_{s, 1, L}^{n,n-s}$ and $L_s^{(2)} = T_{s, 2, L}^{n,n-s}$.  With $num$ standing for the numerator of a maximally reduced fraction, we have
$$gcd(num_{L_s^{(1)}}, num_{L_{s+1}^{(2))}})>1, \qquad s\ge 1.$$  
\end{conjecture}
\begin{remark} The conjecture was tested on several dozen initial values for which it is true. The conjecture has implications for edge values in consecutive reductions of an initial grid. \end{remark}

  This conjecture is a simple example of the plethora of patterns and conjectures that loom in the $n$-grids and their reduction. In fact these reduction patterns of the $n$-grids lie in a 5-dimensional space. The five dimensions are as follows.
\begin{itemize}
    \item $n$ is the number of rows in the initial $n$-grid
    \item $k$ is the number of row reductions that produce an $n-k$-grid under study
    \item The $n-k$ grid is further dimensionalized by
    \begin{itemize}
        \item rows, $r$
        \item diagonals $d$
        \item edges $e$.
    \end{itemize}
\end{itemize}

 The five dimensions allow a richness of perspectives. This paper showed that many patterns abound in these $n$-grids. It is hoped that \cite{Hendel} and this paper will inspire other researchers to delve more deeply into the fascinating patterns that seem to lurk in these $n$-grids. %This furnishes a fertile hunting ground for future results, techniques, and underlying theory. 

\end{document}